\documentclass[12pt]{amsart}

\title{Complete Calabi-Yau metrics on noncompact abelian fibered threefolds}

\author[R. Liang]{Ruiming Liang}
\address[R. Liang]{Peking University}
\email{2201110017@stu.pku.edu.cn}
\author[Y. Zhang]{Yang Zhang}
\address[Y. Zhang]{Nanjing University}
\email{502022210052@smail.nju.edu.cn}
\date{}

\usepackage{array}
\usepackage{amsmath}
\usepackage{mathrsfs}
\usepackage{mathtools}
\usepackage{overpic}
\usepackage{graphicx}
\usepackage{float}
\usepackage[all]{xy}
\usepackage{pifont}
\usepackage{fancyhdr}
\usepackage{amscd}
\usepackage{amsfonts}
\usepackage{amssymb}
\usepackage{amsthm}
\usepackage{bm}
\usepackage{geometry}
\usepackage{setspace}
\usepackage{graphicx}
\usepackage{booktabs}
\usepackage[labelfont=bf]{caption}
\usepackage{caption}
\usepackage{enumitem}
\usepackage{subfigure}
\usepackage{hyperref}

\allowdisplaybreaks[4]

\begin{document}

\onehalfspacing

	\newtheorem{definition}{Definition}
        \newtheorem{theorem}{Theorem}
	\newtheorem{lemma}{Lemma}
	\newtheorem{proposition}{Proposition}
	\newtheorem{remark}{Remark}
	\newtheorem{example}{Example}
        \newtheorem{corollary}{Corollary}

\begin{abstract}

In this article, we construct complete Calabi-Yau metrics on certain type of abelian fibrations $X$ over $\mathbb{C}$. We also provide compactification for $X$ so that the compactified variety has negative canonical bundle.

\end{abstract}
\maketitle

\section{Introduction}

Following Yau's seminal work on the Ricci curvature \cite{yau1978ricci}, the study of Calabi-Yau metrics has become a central topic in Kähler geometry. A significant challenge in this field is to construct complete Calabi-Yau metrics on non-compact Kähler manifolds, a topic that was initiated by Tian and Yau in their paper \cite{tian1990complete}. Their approach involved modeling these metrics on a Fano manifold with its anti-canonical divisor removed.

Following their work, Hein introduced a novel method for constructing such metrics on rational elliptic surfaces, which are not Fano, by removing certain singular fibers \cite{hein2010gravitational}. In these cases, the fibration structure plays a critical role in formulating an explicit ansatz for the metrics. 

In Hein's paper, he mainly focused on the space of a rational elliptic surface $X$ with a special fiber $F$ removed. Recall that rational elliptic surfaces are characterized as smooth surfaces $X$ with relatively minimal elliptic fibrations over $\mathbb{P}^1$ and $K_X=-[F]$. It is clear that $X\backslash F$ has a global holomorphic volume form $\Omega$ with a simple pole along the special fiber $F$. However, it should be pointed out that the condition $K_X=-[F]$ is essential and is closely related to Kodaira's smooth relatively minimal model \cite{kodaira1963compact} and Hein's construction also relies heavily on this condition.

Inspired by Hein's work, we aim to extend this framework to abelian fibrations. In the elliptic fibration case, by Kodaira's canonical bundle formula and Noether formula, we can explicitly figure out the relation between singular fibers and canonical bundle. With some routine argument it is not difficult to see that rational elliptic surfaces are the only plausible playground to construct nontrivial complete Calabi-Yau metrics. However, in the higher dimension case, there is no such explicit numerical formula, so looking for an appropriate geometric model is of some subtlety.

The problem divides into two parts:
\begin{itemize}

\item Construct an abelian fibration over an open Riemann surface and find a noncompact Calabi-Yau metric on it;

\item Compactify the abelian fibration into a projective variety with negative canonical bundle.

\end{itemize}


\begin{definition}[Abelian fibered threefold]

By an abelian fibered threefold, we mean a normal threefold $X$ equipped with a surjective morphism $f:X\to S$ to a Riemann surface $S$ whose general fiber is an abelian surface.

\end{definition}

One of the primary challenges in higher-dimensional geometry arises from the complexities introduced by minimal models.

\begin{definition}[Relatively minimal model]

We call an abelian fibered threefold $f:X\to S$ to be relatively minimal if

(1) $X$ has only $\mathbb{Q}-$factorial terminal singularities;

(2) The canonical Weil divisor $K_X$ is relatively nef, that is, for any curve $C$ in a fiber $X_s$, we have $K_X\cdot C\geq 0$.
  
\end{definition}

\begin{lemma}[Lemma 1.3 in \cite{oguiso1997note}]

Let $f:X\to S$ be a relatively minimal Abelian fibered threefold over a Riemann surface $S$. 
Then there exists a $\mathbb{Q}-$divisor $D$ on $S$ such that $K_X\sim_{\mathbb{Q}} f^*(K_S+D)$.

\end{lemma}

Specifically, for an abelian surface fibered threefold $X$ (not necessarily smooth) over $\mathbb{P}^1$,  while a relatively minimal model $X_{\operatorname{min}}$ does indeed exist, it may not be a smooth variety and lacks uniqueness. Moreover, the canonical divisor $K_X$ in this context is merely a rational multiple of the generic fiber. This situation is detailed further in Crauder and Morrison's work \cite{crauder1994minimal}.

So in order to generalize Hein's construction, we may have to look for fibrations which may not be smooth and whose canonical divisor is a fractional multiple of a fiber.

\begin{remark}

Our definition comes from Oguiso's paper \cite{oguiso1997note}. In the paper \cite{crauder1994minimal}, the authors also give a definition of minimal ablian fibration. It is clear that their definition is stronger than Oguiso's definition.

\end{remark}

Now we state the main results of our paper (details to be discussed in the following sections).

\begin{definition}

Let $N$ be an open manifold, $g$ be a complete Riemannian metric on $N$. 

(a) The metric $g$ is $\mathrm{ALG}
$, if there exists $r_0, \delta, \theta>0$, a compact set $K\subset N$ and a diffeomorphism $\Phi:S(\theta,r_0)\times T^{2m}\hookrightarrow N\backslash K$ with dense image such that
\begin{equation*}
\left|\nabla^k_{\mathrm{flat}}(\Phi^*g-g_{\mathrm{flat}})\right|_{g_{\mathrm{flat}}}\leq C(k)|z|^{-\delta-k}, \; \forall k\in\mathbb{N},
\end{equation*}where $S(\theta,r_0)=\{z\in\mathbb{C}:|z|>r_0,0<\arg z<2\pi \theta\} $ and $g_{\mathrm{flat}}=g_{\mathbb{C}}\oplus g_T$ for some flat metric on $T^{2m}$.

(b) The metric $g$ is $\mathrm{ALH}$ if there exists $\delta>0$,  a compact set $K\subset N$, and a diffeomorphism $\Phi:\mathbb{R}^+\times T^{2m+1}\to N\backslash K$ such that
\begin{equation*}
\left|\nabla^k_{g_{\mathrm{flat}}}\left(\Phi^*g-g_{\mathrm{flat}}\right)\right|_{g_{\mathrm{flat}}}\leq C(k)e^{-\delta t},\; \forall k\in\mathbb{N},
\end{equation*}where $g_{\mathrm{flat}}=dt^2+h$ with $h$ some flat metric on $T^{2m+1}$. 

\end{definition}

\begin{remark}

(1) This definition of ALG and ALH manifold is a direct generalization of the one in Hein's thesis. Also one can consider more general definitions, such as replacing the torus factor with an arbitrary compact Calabi-Yau manfold. 

(2) The tangent cones of ALH and ALG manifolds at infinity are $\mathbb{R}$ and $C$ respectively, where $C$ is a flat $2$-dimensional metric cone.

(3) We will use the concepts $\operatorname{ALG}^*$ and $\operatorname{ALH}^*$ without definition in our main theorem below, for the precise meaning of these two words, see the more precise version of the main Theorem \ref{main 2} in Section \ref{non-isotrivial models}.

\end{remark}

Our main theorem is:

\begin{theorem}[Main theorem]\label{main 0}

Let $\pi_i:X_i\to \mathbb{P}^1$, $i=1,2$ be two rational elliptic fibrations. Specify one particular singular fiber $F_i$ on each $X_i$, $i=1,2$. Let the two singular fibers $F_1,F_2$ project to the same point on $\mathbb{P}^1$ while all the other singular fibers to different points. Take the fiber product $f:X=X_1\times_{\mathbb{P}^1}X_2\to \mathbb{P}^1$. Now $X$ is a singular variety with all the singular loci contained in the singular fiber $F=F_1\times F_2$. Let $M=X\backslash F$. 

(1) Suppose the monodromy of both $F_1$ and $F_2$ are finite. Then for some specific types, there could exist $\operatorname{ALG}$ or $\operatorname{ALH}$ type Calabi-Yau metric on $M$.

(2) If $F_i$ is of $I_{b_i}^*$ type, then there exists generalized $\operatorname{ALH}^*$ type Calabi-Yau metric on $M$.

(3) If $F_1$ is of $I_{b}^*$ type and $F_2$ is of $\mathrm{II}^*$ or $\mathrm{III}^*$ or $\mathrm{IV}^*$ type, then there exists generalized $\operatorname{ALG}^*$ type Calabi-Yau metric on $M$.

\end{theorem}

This paper is organized as follows:
\begin{itemize}

\item In Section \ref{torus bundles}, we recall the semi-flat metrics on smooth torus bundles. 

\item In Section \ref{isotrivial models}, we discuss an isotrivial construction, which will shed some light on the construction of the non-isotrivial one in Section \ref{non-isotrivial models}.

\item In Section \ref{non-isotrivial models}, we focus on the construction of ansatz at infinity on non-isotrivial abelian surface fibrations and discuss the asymptotical behavior at infinity. 

\item In Section \ref{perturbation}, we prove a $\partial\bar{\partial}$-lemma in our cases. Then we glue the local ansatz constructed at infinity in Section \ref{non-isotrivial models} with an appropriate initial k{\"a}hler form on the whole space to get a background metric for the complex Monge-Ampère equation. Then with Tian-Yau-Hein's package, we can perturb the background metric into a genuine Calabi-Yau metric whose behavior at infinity has already been well understood.

\item In Section \ref{Further discussion}, we discuss some generalizations of our construction and the potential difficulties we may encounter.

\item In Appendix \ref{Tian-Yau-Hein's package}, we briefly review the Tian-Yau-Hein's package.

\end{itemize}

\textbf{Acknowledgments:}  We would like to express our sincere gratitude to Professor Gang Tian, for his helpful suggestions, patient guidance and revision of the earlier version. We are deeply thankful to Professor Yalong Shi for his helpful discussions. Finally, we are very appreciative of Chenhan Liu for his 
encouragement and assistance throughout the research. Yang Zhang is supported by NSFC No.12371058.

\section{Review of Ricci-flat metrics on torus bundles}\label{torus bundles}

In order to explain our construction, we review some basic facts following Hein's thesis \cite{hein2010gravitational}.

Assume $f: X \rightarrow S$ is a holomorphic submersion over a Riemann surface S (not necessarily compact) such that all fibers $X_s=f^{-1}(s)$ are smooth complex $m$-tori. Assume $f$ admits a holomorphic section $\sigma:S\to X$ and a constant polarization $\omega$, that is $\omega\in \Omega^2(X,\mathbb{R})$ such that $\omega|_{X_s}$ is a K{\"ahler} metric, and there exists $c_i \in \mathbb{R}, \eta_i(s) \in H^2\left(X_s, \mathbb{Z}\right)$ such that $\left[\left.\omega\right|_{X_s}\right]=\sum c_i \eta_i(s)$ for all $s \in S$.

From the lemma below, we know that when $\omega$ is a closed form, we get a constant polarization.

\begin{lemma}[Proposition 2.1 in \cite{hein2015remarks}]\label{Cartan lee}

Suppose $X_s\cong \mathbb{C}^m/\Lambda_s$. Fix a basis $\left(\tau_1(s), \ldots, \tau_{2 m}(s)\right)$ of the lattice $\Lambda_s$ that varies holomorphically with $s \in S$, and let $\left(\xi^1(s), \ldots, \xi^{2 m}(s)\right)$ be the $\mathbb{R}$-dual basis of 1-forms on $\mathbb{C}^m$. Then the classes $\left[\xi^i(s) \wedge \xi^j(s)\right] \in$ $H^2\left(X_s, \mathbb{Z}\right) \subset H^2\left(X_s, \mathbb{R}\right)$ with $i<j$ form a basis of $H^2\left(X_s, \mathbb{Z}\right)$. Let $\omega$ be a real 2-form on $X$ whose restriction to $X_s$ is closed for all $s \in S$, and expands $\left[\left.\omega\right|_{X_s}\right]=\sum\limits_{i<j} P_{i j}(s)\left[\xi^i(s) \wedge \xi^j(s)\right]$. If $\omega$ is closed, then $P_{i j}(s)$ do not depend on $y$.

\end{lemma}

\begin{proof}

Consider the Gauss-Manin connection $\nabla^{\mathrm{GM}}$ on the smooth $\mathbb{R}$-vector bundle $R^2 f_* \mathbb{R} \otimes \mathcal{C}_S^{\infty}$ on $S$. By definition, the sections $\left[\xi^i(y) \wedge \xi^j(y)\right]$ of this vector bundle form a basis of the space of $\nabla^{\mathrm{GM}}$-parallel sections. On the other hand, since $\omega$ is closed, Cartan's magic formula yields that: 
\begin{equation*}
    \nabla^{\mathrm{GM}}_v\left[\left.\omega\right|_{X_s}\right]=\left[d\circ \iota_{\tilde{v}}\omega+\iota_{\tilde{v}}\circ d\omega\right]=\left[\iota_{\tilde{v}}\circ d\omega\right]=0,
\end{equation*}where $\tilde{v}$ is the horizontal lift of tangent vector $v$ on $S$ with respect to the Gauss-Manin connection, and the second equality comes from the assumption that $\omega|_{X_s}$ is closed.

\end{proof}

For $\varepsilon>0, s \in S$, we denote by $g_{s, \varepsilon}$ the unique flat K{\"a}hler metric on $X_s$ with $\operatorname{vol}\left(X_s, g_{s, \varepsilon}\right)=\varepsilon$ and whose K{\"a}hler class is a proportional to  $\left[\left.\omega\right|_{X_s}\right]$. Restricting $g_{s, \varepsilon}$ to the tangent space $T_{\sigma(s)} X_s$ induces a hermitian metric $h_{\varepsilon}$ on the holomorphic vector bundle $E:=\sigma^* T_{X / S}$, that is 
\begin{equation*}
h_{\varepsilon,s}(v,w)\coloneqq g_{s,\varepsilon}(v,w),\quad \forall v,w\in T_{\sigma(s)}X_s=(\sigma^*T_{X/S})_s=E_s.
\end{equation*}

Assume further that $X$ has a holomorphic volume form $\Omega$. Then $\Omega$ induces a faithful pairing 
\begin{equation*}
\wedge^m E \otimes T^{1,0} S \rightarrow \mathbb{C},\; \left((v_1,...,v_m),w\right)\mapsto \Omega(v_1,...,v_m,\sigma_*w)
\end{equation*}
 Using the hermitian metric on $\wedge^m E$ induced by $h_{\varepsilon}$, we get a Riemannian metric $g_{S, \varepsilon}$. Explicitly, this metric is given by:
\begin{equation*}
g_{S,\varepsilon}(w,w)=\frac{\Omega\wedge \bar{\Omega}(v_1,...,v_n,\sigma_*w,\bar{v}_1,...,\bar{v}_n,\sigma_*\bar{w})}{\det\left(h_\varepsilon(v_i,v_j)\right)}. 
\end{equation*}

Using $g_{S, \varepsilon}$ and the family of flat metrics $\left\{g_{s, \varepsilon}\right\}$ on the fibers, we can build a submersion metric on $X$.

\begin{lemma}

Choose a fiber-preserving biholomorphism $X \cong E / \Lambda$ for some holomorphic lattice bundle $\Lambda \subset E$.

(1) $\Lambda$ induces a flat $\mathbb{R}$-linear connection on $E$, hence an integrable horizontal distribution $\mathcal{H}$ on $X$

(2) $\mathcal{H}$ is independent of the biholomorphism $X \cong E / \Lambda$.

\end{lemma}

\begin{proof}

(1) Define $\nabla:E\to E\otimes \Omega_S$ as follows: For any $\sigma\in E$, we can write $\sigma=\sum\limits_i\alpha_i\sigma_i$, where $\sigma_i\in \Lambda$ form a local basis of $E$. Set
\begin{equation*}
\nabla \sigma=\sum\limits_{i}\sigma_i\otimes d\alpha_i\in E\otimes \Omega_S.
\end{equation*}

(2) Two such isomorphisms differ only by a holomorphic section $S \rightarrow \operatorname{Aut}(E)$.

\end{proof}

\begin{remark}

The above connection is none other than the Gauss-Manin connection. Specifically, consider the sheaf $R^1f_*\mathbb{Z}$ on $S$. Since $f:X\to S$ is a proper submersion, $R^1f_*\mathbb{Z}$ is locally constant with $\left(R^1f_*\mathbb{Z}\right)\big{|}_{s}=H^1(X_s,\mathbb{Z})=(T_{\sigma(s)}X_s)^{\vee}$. Then $E^{\vee}\cong (R^1f_*\mathbb{Z})\otimes\mathcal{C}^\infty(S)$. Thus the connection in the above lemma is just the dual connection of the Gauss-Manin connection on the locally free sheaf $(R^1f_*\mathbb{Z})\otimes\mathcal{C}^\infty(S)$.

\end{remark}

Define
\begin{equation}\label{semi-flat 0}
g_{\mathrm{sf}, \varepsilon}(u, v):=g_{S, \varepsilon}\left(f_* u, f_* v\right)+g_{s, \varepsilon}(P u, P v),\; u, v \in T_x X, s=f(x),
\end{equation}where $P:=T_xX\to \mathcal{V}_x=T_xX_s$ is the projection along $\mathcal{H}_x$. This defines a semi-flat hermitian metric $g_{\mathrm{sf}, \varepsilon}$ on $X$. 

For computational purpose, we now derive the  explicit formula for $g_{\mathrm{sf}, \varepsilon}$ in local coordinates in terms of $\Omega$ and the periods of the tori $X_s$ with respect to $\omega$. 

Fix isomorphisms $\left.\left.X\right|_U \cong E\right|_U \big{/}\Lambda \cong\left(U \times \mathbb{C}^m\right) / \Lambda$ over an open set $U \subset  \mathbb{C}$, with coordinates $z$ on $U$ and $v$ along the fibers. Fix an oriented basis $\left(\tau_1, \ldots, \tau_{2 m}\right)$ for $\Lambda$ at one point and extend it to a tuple of multi-valued functions $\tau_i(z) \in \mathcal{O}\left(U, \mathbb{C}^m\right)$ that generates $\Lambda$ everywhere. Let $\left(\xi^1(z), \ldots, \xi^{2 m}(z)\right)$ be $\mathbb{R}$-dual to $\left(\tau_1(z), \ldots, \tau_{2 m}(z)\right)$. Now we can relate the complex coordinate $v_1,...,v_m$ to the real basis: Let $\tau_i(z)=(T_{1i}(z),...,T_{mi}(z))$, $\forall i=1,...,2m$. Define
\begin{equation*}
T=\begin{pmatrix}
T_{11} & \cdots & T_{12m}\\
\vdots & \ddots & \vdots\\
T_{m1} & \cdots & T_{m2m}
\end{pmatrix}=(\tau_1^t,...,\tau_{2m}^t)
\in \mathcal{O}\left(U, \mathbb{C}^{m \times 2 m}\right).
\end{equation*}
Then
\begin{align*}
(dv^1,...,dv^m)=(\xi^1,...,\xi^{2m})\begin{pmatrix}
T^{11}& \cdots & T_{m1}\\
\vdots & \ddots & \vdots\\
T_{12m}& \cdots & T_{m2m}
\end{pmatrix}.
\end{align*}Written briefly, we have
\begin{align*}
\begin{pmatrix}
dv\\
d\bar{v}
\end{pmatrix}=\begin{pmatrix}
T\\
\bar{T}
\end{pmatrix}\cdot \xi,\quad \xi=(
\Pi, \bar{\Pi})\cdot\begin{pmatrix} dv\\
d\bar{v}
\end{pmatrix},
\end{align*}where $\Pi\in \mathcal{O}\left(U, \mathbb{C}^{2m \times m}\right)$ satisfy $(\Pi,\bar{\Pi})\begin{pmatrix}
T\\
\bar{T}
\end{pmatrix}=I_{2m}$.

Existence of a constant polarization is equivalent to the existence of $Q \in \mathbb{R}^{2 m \times 2 m}, Q+Q^{\operatorname{tr}}=0$, such that $\omega=\frac{1}{2} \sum Q_{i j} \xi^i \wedge \xi^j$ restricts to a flat K{\"a}hler metric on each fiber.

\begin{lemma}

(1) Since $T$ is multi-valued, the well-definedness of $\omega$ requires that under the monodromy action, $T$ transforms as $T\mapsto TA$, where $A \in \operatorname{Gl}(2 m, \mathbb{R})$ satisfies $A^{\operatorname{tr}} Q A=Q$.

(2) There exists $S \in \operatorname{Gl}(2 m, \mathbb{R})$ such that $S^{\operatorname{tr}} Q S=\begin{pmatrix}
0 & I \\ 
-I & 0
\end{pmatrix}
$.

(3) Write $T S=R(1, Z)$ with $R: U \rightarrow \mathrm{Gl}(m, \mathbb{C}), Z: U \rightarrow \mathbb{C}^{m \times m}$ multi-valued holomorphic maps. Then $\omega$ is positive $(1,1)$ if and only if $Z$ belongs to the Siegel upper half-plane $\mathfrak{H}_m$

(4) $\omega=i H_{j k} d v^j \wedge d \bar{v}^k$, where $H^{-1}:=2 \bar{R}(\operatorname{Im} Z) R^{\operatorname{tr}}=i \bar{T} Q^{-1} T^{\operatorname{tr}}$.

\end{lemma}

\begin{proof}

We only prove (4). First prove $2\bar{R}(\operatorname{Im}Z)R^t=\sqrt{-1}\bar{T}Q^{-1}T^t$. Note
\begin{align*}
\frac{\sqrt{-1}}{2}\overline{R^{-1}T}Q^{-1}(R^{-1}T)^t&=\frac{\sqrt{-1}}{2}(1,\bar{Z})S^{-1}Q^{-1}(S^{-1})^t(1,Z)^t\\
&=\frac{\sqrt{-1}}{2}(1,\bar{Z})\left(S^tQS\right)^{-1}(1,Z)^t\\
&=-\frac{\sqrt{-1}}{2}(1,\bar{Z})\begin{pmatrix}
0& I\\
-I& 0
\end{pmatrix}(1,Z)^t\\
&=\operatorname{Im}Z.
\end{align*}

Let $\omega=\frac{1}{2} \sum Q_{i j} \xi^i \wedge \xi^j$ and $\tilde{\omega}=i H_{j k} d v^j \wedge d \bar{v}^k$.

Change variable by $\left(\eta^1,...,\eta^{2m}\right)=\left(\xi^1,...,\xi^{2m}\right)(S^t)^{-1}$. Then $\omega=\frac12\sum\tilde{Q}_{ij}\eta^i\wedge \eta^j$, where $\left(\tilde{Q}_{ij}\right)=\begin{pmatrix}
0 & I\\
-I & 0
\end{pmatrix}$. Now $\tilde{S}=I_{2m}$. 

Change variable by $\left(dz^1,...,dz^m\right)=\left(dv^1,...,dv^m\right)(R^{-1})^t$, then
\begin{equation*}
\tilde{\omega}=\frac{\sqrt{-1}}{2}\left(dz^1,...,dz^m\right)(\operatorname{Im}Z)^{-1}\begin{pmatrix}
dz^1\\
\vdots\\
dz^m
\end{pmatrix}.
\end{equation*}Now $\tilde{R}=1$, $\tilde{T}=(1,Z)$. Hence to prove $\omega=\tilde{\omega}$, we only need to consider the case $S=1$, $R=1$, $Q=\begin{pmatrix}
0 & I\\
-I & 0
\end{pmatrix}$.

Let $\Pi=(\pi_{i\alpha})_{2m\times m}$ such that 
\begin{align*}
\xi^i=\sum_\alpha \pi_{i \alpha} d v_\alpha+\sum_\alpha \bar{\pi}_{i \alpha} d \bar{v}_\alpha.
\end{align*}
Now direct computation shows
\begin{align*}
\omega&=\frac{1}{2} \sum q_{i j} \xi^i \wedge \xi^j\\
&= \frac{1}{2} \sum q_{i j}\left(\pi_{i \alpha} d v_\alpha+\bar{\pi}_{i \alpha} d \bar{v}_\alpha\right) \wedge\left(\pi_{j \beta} d v_\beta+\bar{\pi}_{j \beta} d \bar{v}_\beta\right) \\
&= \frac{1}{2} \sum q_{i j} \pi_{i \alpha} \pi_{j \beta} d v_\alpha \wedge d v_\beta+\frac{1}{2} \sum q_{i j} \bar{\pi}_{i \alpha} \bar{\pi}_{j \beta} d \bar{v}_\alpha \wedge d \bar{v}_\beta \\
& +\frac{1}{2} \sum q_{i j}\left(\pi_{i \alpha} \bar{\pi}_{j \beta}-\bar{\pi}_{i \beta} \pi_{j \alpha}\right) d v_\alpha \wedge d \bar{v}_\beta .
\end{align*}
From this we see that $\omega$ is of type $(1,1)$ if and only if the coefficient matrix
\begin{align*}
\frac{1}{2}\left(\sum_{i, j} q_{i j} \pi_{i \alpha} \pi_{j \beta}\right)= \Pi^t \cdot Q \cdot \Pi=0;
\end{align*}
and $\omega$ is positive if and only if
\begin{align*}
\frac{1}{2 \sqrt{-1}}\left(\sum_{i, j} q_{i j}\left(\pi_{i \alpha} \bar{\pi}_{j \beta}-\bar{\pi}_{i \beta} \pi_{j \alpha}\right)\right)_{\alpha \beta} \quad=\frac{1}{2 \sqrt{-1}}\left( \Pi^t Q \bar{\Pi}- \Pi^t Q \bar{\Pi}\right)=\frac{1}{\sqrt{-1}} \Pi^t Q \bar{\Pi}>0
\end{align*}
is hermitian positive definite. 

Now $Q=\begin{pmatrix}
0 & 1\\
-1 & 0
\end{pmatrix}$ and $(\Pi,\bar{\Pi})\begin{pmatrix}
T\\
\bar{T}
\end{pmatrix}=I_{2m}$ implies that $\Pi=\begin{pmatrix}
\Pi_1\\
\Pi_2
\end{pmatrix}$, where $\Pi_1=-\frac{\sqrt{-1}}{2}(\operatorname{Im}Z)^{-1}$, $\Pi_2=\frac{\sqrt{-1}}{2}\bar{Z}(\operatorname{Im}Z)^{-1}$. Hence
\begin{align*}
\omega=\frac{1}{2} \sum q_{i j} \xi^i \wedge \xi^j=(dv^1,...,dv^m)\Pi^tQ\bar{\Pi}\begin{pmatrix}
dv^1\\
\vdots\\
dv^m
\end{pmatrix}.
\end{align*}Note that
\begin{align*}
\Pi^tQ\bar{\Pi}=\left(-\frac{\sqrt{-1}}{2}(\operatorname{Im}Z)^{-1},\frac{\sqrt{-1}}{2}\bar{Z}(\operatorname{Im}Z)^{-1}\right)\begin{pmatrix}
0 & 1\\
-1 & 0
\end{pmatrix}\overline{\begin{pmatrix}
-\frac{\sqrt{-1}}{2}(\operatorname{Im}Z)^{-1}\\
\frac{\sqrt{-1}}{2}\bar{Z}(\operatorname{Im}Z)^{-1}
\end{pmatrix}}=-\frac{\sqrt{-1}}{2}(\operatorname{Im}Z)^{-1}.
\end{align*}
Therefore, $\omega=\tilde{\omega}$.

\end{proof}

Note that
\begin{align*}
\nabla_{\partial z_i}\begin{pmatrix}
dv^1\\
\vdots\\
dv^m
\end{pmatrix}=\nabla_{\partial z_i}T\cdot \begin{pmatrix}
\xi^1\\
\vdots\\
\xi^{2m}
\end{pmatrix}=\frac{\partial T}{\partial z_i}\begin{pmatrix}
T\\
\bar{T}
\end{pmatrix}^{-1}\begin{pmatrix}
dv\\
d\bar{v}
\end{pmatrix}
\end{align*}
Therefore, the flat connection induced by $\Lambda$ has Christoffel symbols
\begin{align*}
\Gamma_i(z, v)=\frac{\partial T}{\partial z^i}\binom{T}{\bar{T}}^{-1}\binom{v}{\bar{v}} \in \mathbb{C}^m, \quad i=1, \ldots, \operatorname{dim}_C U
\end{align*}
i.e. the vectors $\left(e_i, \Gamma_i(z, v)\right)$ span the horizontal space at $(z, v) \in U \times \mathbb{C}^m$.

\begin{theorem}\label{semi-flat 1}

Define $H(\varepsilon):=\left(\frac{\varepsilon}{\sqrt{\operatorname{det} Q}}\right)^{\frac{1}{m}} H$, where $H^{-1}:=2 \bar{R}(\operatorname{Im} Z) R^t=i \bar{T} Q^{-1} T^t$. 

(1) If $\Omega=g d z \wedge d v^1 \wedge \ldots \wedge d v^m$ for some holomorphic function $g: U \rightarrow \mathbb{C}$, then the K{\"a}hler form of the hermitian metric $g_{\mathrm{sf}, \varepsilon}$ is given by
\begin{align}\label{semi-flat 2}
\omega_{\mathrm{sf}, \varepsilon}=i|g|^2 \operatorname{det}(H(\varepsilon))^{-1} d z \wedge d \bar{z}+i H(\varepsilon)_{j k}\left(d v^j-\Gamma^j d z\right) \wedge\left(d \bar{v}^k-\bar{\Gamma}^k d \bar{z}\right).
\end{align}
This is a closed form with top power $(m+1)!i^{(m+1)^2} \Omega \wedge \bar{\Omega}$. In particular, $g_{\mathrm{sf}, \varepsilon}$ is a Calabi-Yau metric.

(2) The metric $\omega_{S, \varepsilon}$ on $S$ is K{\"a}hler, and the Ricci form satisfies
\begin{align}\label{semi-flat 3}
\rho\left(\omega_{S, \varepsilon}\right)=-i \partial \bar{\partial} \log \operatorname{det}(\operatorname{Im} Z)=\frac{i}{4}\left((\operatorname{Im} Z)^{a b} d Z_{b c} \wedge(\operatorname{Im} Z)^{c d} d \bar{Z}_{d a}\right)
\end{align}

\end{theorem}

\begin{corollary}\label{semi-flat 4}

Let $f: X \rightarrow S$ be an elliptic fibration without singular fibers over a Riemann surface, equipped with a holomorphic section $\sigma$. Consider a holomorphic volume form $\Omega$ on $X$ and let $\omega_{\mathrm{sf}, \varepsilon}$ be the semi-flat hyperk{\"a}hler metric on $X$ constructed from $\sigma$ and $\frac{1}{\sqrt{2}} \Omega$ such that the fibers have $\omega_{\mathrm{sf}, \varepsilon}$-area equal to $\varepsilon$. Thus, $\omega_{\mathrm{sf}, \varepsilon}^2=\Omega \wedge \bar{\Omega}$.

Given a domain $U\subset S$ with a holomorphic coordinate $z$, and fixing an isomorphism of elliptic fibrations $\left.X\right|_U \cong\left(U \times \mathbb{C}_v\right) /\left(\mathbb{Z} \tau_1+\mathbb{Z} \tau_2\right)$, where $\tau_1, \tau_2: U \rightarrow \mathbb{C}$ are multi-valued holomorphic functions ensuring that $\left(\tau_1, \tau_2\right)$ is positively oriented and $\sigma$ maps to the zero section. Then $\Omega=g d z \wedge d v$ for some holomorphic function $g: U \rightarrow \mathbb{C}$, and
\begin{align}\label{semi-flat 5}
& \omega_{\mathrm{sf}, \varepsilon}=i|g|^2 \frac{\operatorname{Im}\left(\bar{\tau}_1 \tau_2\right)}{\varepsilon} d z \wedge d \bar{z}+\frac{i}{2} \frac{\varepsilon}{\operatorname{Im}\left(\bar{\tau}_1 \tau_2\right)}(d v-\Gamma d z) \wedge(d \bar{v}-\bar{\Gamma} d \bar{z}); \\
& \Gamma(z, w)=\frac{1}{\operatorname{Im}\left(\bar{\tau}_1 \tau_2\right)}\left(\operatorname{Im}\left(\bar{\tau}_1 v\right) \tau_2^{\prime}-\operatorname{Im}\left(\bar{\tau}_2 v\right) \tau_1^{\prime}\right).
\end{align}

\end{corollary}

Especially in our cases, we have

\begin{corollary}\label{semi-flat 6}

Let $f: X \rightarrow S$ be the fiber product of two elliptic fibrations without singular fibers over a Riemann surface, equipped with a holomorphic section $\sigma$. Consider a holomorphic volume form $\Omega$ on $X$ and let $\omega_{\mathrm{sf}, \varepsilon}$ be the semi-flat Calabi-Yau metric on $X$ constructed from $\sigma$ and $ \Omega$ such that the fibers have $\omega_{\mathrm{sf}, \varepsilon}$-area equal to $\varepsilon$. Thus, in particular, $\omega_{\mathrm{sf}, \varepsilon}^3=6i\Omega \wedge \bar{\Omega}$. 
 
 Let $U$ be a domain in $S$, let $z$ be a holomorphic coordinate on $U$, and fix an isomorphism of elliptic fibrations $\left.X\right|_U \cong\left(U \times \mathbb{C}_{v_1}\times \mathbb{C}_{v_2}\right) /\left(\mathbb{Z} \tau_1+\mathbb{Z} \tau_2+\mathbb{Z} \tau_3+\mathbb{Z} \tau_4\right)$ with multi-valued holomorphic functions $\tau_1, \tau_2,\tau_3,\tau_4: U \rightarrow \mathbb{C}$ so that $\left(\tau_1, \tau_2,\tau_3,\tau_4\right)$ is positively oriented and $\sigma$ maps to the zero section. Then $\Omega=g d z \wedge d v_1\wedge dv_2$ for some holomorphic function $g: U \rightarrow \mathbb{C}$, and
\begin{equation}\label{semi-flat 7}
\begin{aligned}
\omega_{\mathrm{sf}, \varepsilon}&=4i|g|^2 \frac{\operatorname{Im}\left(\bar{\tau}_1 \tau_2\right)\operatorname{Im}\left(\bar{\tau}_3 \tau_4\right)}{\varepsilon^2} d z \wedge d \bar{z}\\
&+\frac{i}{2} \frac{\varepsilon}{\operatorname{Im}\left(\bar{\tau}_1 \tau_2\right)}(d v_1-\Gamma^1 d z) \wedge(d \bar{v}_1-\bar{\Gamma}^1 d \bar{z})\\
&+\frac{i}{2} \frac{\varepsilon}{\operatorname{Im}\left(\bar{\tau}_3 \tau_4\right)}(d v_2-\Gamma^2 d z) \wedge(d \bar{v}_2-\bar{\Gamma}^2 d \bar{z})
\end{aligned}
\end{equation}
where
\begin{equation}\label{semi-flat 8}
\begin{aligned}
& \Gamma^1(z, v)=\frac{1}{\operatorname{Im}\left(\bar{\tau}_1 \tau_2\right)}\left(\operatorname{Im}\left(\bar{\tau}_1 v_1\right) \tau_2^{\prime}-\operatorname{Im}\left(\bar{\tau}_2 v_1\right) \tau_1^{\prime}\right);\\
&\Gamma^2(z, v)=\frac{1}{\operatorname{Im}\left(\bar{\tau}_3 \tau_4\right)}\left(\operatorname{Im}\left(\bar{\tau}_3 v_2\right) \tau_4^{\prime}-\operatorname{Im}\left(\bar{\tau}_4 v_2\right) \tau_3^{\prime}\right).
\end{aligned}
\end{equation}

\end{corollary}

\begin{proof}

In this case $T=\begin{pmatrix}
\tau_1 &\tau_2 & 0 & 0\\
0 & 0 & \tau_3 & \tau_4
\end{pmatrix}.$
Note
\begin{align*}
\begin{pmatrix}
\tau_1 &\tau_2 & 0 & 0\\
0 & 0 & \tau_3 & \tau_4.
\end{pmatrix}\cdot \begin{pmatrix}
1 & 0 & 0 & 0\\
0 & 0 & 1 & 0\\
0 & 1 & 0 & 0\\
0 & 0 & 0 & 1\\
\end{pmatrix}=\begin{pmatrix}
\tau_1 & 0\\
0 & \tau_3
\end{pmatrix}\begin{pmatrix}
1 & 0 & \frac{\tau_2}{\tau_1} & 0\\
0 & 1 & 0 &\frac{\tau_4}{\tau_3}
\end{pmatrix}.
\end{align*}
Thus
\begin{align*}
S=\begin{pmatrix}
1 & 0 & 0 & 0\\
0 & 0 & 1 & 0\\
0 & 1 & 0 & 0\\
0 & 0 & 0 & 1\\
\end{pmatrix},\; R=\begin{pmatrix}
\tau_1 & 0\\
0 & \tau_3
\end{pmatrix},\; Z=\begin{pmatrix}
 \frac{\tau_2}{\tau_1} & 0\\
0 &\frac{\tau_4}{\tau_3}
\end{pmatrix}.
\end{align*}
Hence
\begin{align*}
H^{-1}=\begin{pmatrix}
2\operatorname{Im}(\bar{\tau}_1\tau_2) & 0\\
0 & 2\operatorname{Im}(\bar{\tau}_3\tau_4)
\end{pmatrix},\; H=\begin{pmatrix}
\frac{1}{2\operatorname{Im}(\bar{\tau}_1\tau_2)} & 0\\
0 & \frac{1}{2\operatorname{Im}(\bar{\tau}_3\tau_4)}
\end{pmatrix},\; \operatorname{det}H^{-1}=4\operatorname{Im}\left(\bar{\tau}_1 \tau_2\right)\operatorname{Im}\left(\bar{\tau}_3 \tau_4\right)
\end{align*}
By Gaussian Elimination algorithm, we can get the inverse matrix
\begin{align*}
\begin{pmatrix}
T\\
\bar{T}
\end{pmatrix}^{-1}=\begin{pmatrix}
\frac{i\bar{\tau}_2}{2\operatorname{Im}(\bar{\tau}_1\tau_2)} & 0 & \frac{-i\tau_2}{2\operatorname{Im}(\bar{\tau}_1\tau_2)} & 0\\
\frac{-i\bar{\tau}_1}{2\operatorname{Im}(\bar{\tau}_1\tau_2)} & 0 & \frac{i\tau_1}{2\operatorname{Im}(\bar{\tau}_1\tau_2)} & 0\\
0 & \frac{i\bar{\tau}_4}{2\operatorname{Im}(\bar{\tau}_3\tau_4)} & 0 & \frac{-i\tau_4}{2\operatorname{Im}(\bar{\tau}_3\tau_4)}\\
0 & \frac{-i\bar{\tau}_3}{2\operatorname{Im}(\bar{\tau}_1\tau_2)} & 0 & \frac{i\tau_3}{2\operatorname{Im}(\bar{\tau}_3\tau_4)}\\
\end{pmatrix}
\end{align*}
Then by direct computation, we get
\begin{align*}
\Gamma(z,w)&=\begin{pmatrix}
\tau_1' &\tau_2' & 0 & 0\\
0 & 0 & \tau_3' & \tau_4'.
\end{pmatrix}\begin{pmatrix}
\frac{i\bar{\tau}_2}{2\operatorname{Im}(\bar{\tau}_1\tau_2)} & 0 & \frac{-i\tau_2}{2\operatorname{Im}(\bar{\tau}_1\tau_2)} & 0\\
\frac{-i\bar{\tau}_1}{2\operatorname{Im}(\bar{\tau}_1\tau_2)} & 0 & \frac{i\tau_1}{2\operatorname{Im}(\bar{\tau}_1\tau_2)} & 0\\
0 & \frac{i\bar{\tau}_4}{2\operatorname{Im}(\bar{\tau}_3\tau_4)} & 0 & \frac{-i\tau_4}{2\operatorname{Im}(\bar{\tau}_3\tau_4)}\\
0 & \frac{-i\bar{\tau}_3}{2\operatorname{Im}(\bar{\tau}_1\tau_2)} & 0 & \frac{i\tau_3}{2\operatorname{Im}(\bar{\tau}_3\tau_4)}\\
\end{pmatrix}\begin{pmatrix}
v_1\\
v_2\\
\bar{v}_1\\
\bar{v}_2
\end{pmatrix}\\
&=\begin{pmatrix}
&\frac{1}{\operatorname{Im}\left(\bar{\tau}_1 \tau_2\right)}\left(\operatorname{Im}\left(\bar{\tau}_1 v_1\right) \tau_2^{\prime}-\operatorname{Im}\left(\bar{\tau}_2 v_1\right) \tau_1^{\prime}\right)\\
&\frac{1}{\operatorname{Im}\left(\bar{\tau}_3 \tau_4\right)}\left(\operatorname{Im}\left(\bar{\tau}_3 v_2\right) \tau_2^{\prime}-\operatorname{Im}\left(\bar{\tau}_4 v_2\right) \tau_3^{\prime}\right)
\end{pmatrix}=\begin{pmatrix}
    \Gamma^1(z,v)\\
    \Gamma^2(z,v)
\end{pmatrix}.
\end{align*}
Plugging all these data into the semi-flat metric formula, we get the desired formula.

\end{proof}

\section{Construction of isotrivial fibrations}\label{isotrivial models}


In this section, we discuss the isotrivial construction on a quotient model.

\begin{theorem}\label{main 1}

Let $A$ be an abelian surface. Let $\mathbb{Z}/k\mathbb{Z}$ acts on $\mathbb{P}^1_{[t:s]}\times A$ in such a manner that

(a) $\mathbb{Z}_k$ acts on $\mathbb{P}^1$ by $t\mapsto \zeta_k^{-1} t$, so that the fixed loci are at $t=0$ and $s=0$, where $\zeta_k=\exp\left(\frac{2\pi i}{k}\right)$ is a primitive $k$-th unit root;

(b) $\mathbb{Z}_k$ induces an automorphism $\tau$ of $A$ such that $\tau^*\Omega_A=\zeta_k\Omega_A$, where $\Omega_A$ is the holomorphic volume form on $A$.

Let $X$ be the crepant resolution of the quotient space $(\mathbb{P}^1_{[t:s]}\times A)/\mathbb{Z}_k$ at $t=0$. There is a natural abelian fibration stucture $f:X\to \mathbb{P}^1$. Let $F=f^*(\{\infty\})$ (viewed as the pullback of a divisor) be the special fiber at $t=\infty$. Then

(1) $F$ is an irreducible but non reduced divisor with multiplicity $k$;

(2) $K_X=-\frac{2}{k}[F]$ and there exists a holomorphic volume form $\Omega$ on $M=X\backslash F$ with a second order pole on the reduced space $F_{\operatorname{red}}$.

(3) For significant large $\lambda$, there exists an appropriately chosen explicit initial metric $\omega_0$ on the non compact manifold $M$ and a smooth function $u$ with Schwartz decay such that the metric
\begin{equation*}
\omega_{\operatorname{CY}}=\omega_0+i\partial\bar{\partial}u
\end{equation*}satisfies the complex Monge-Ampère equation
\begin{equation*}
\omega_{\operatorname{CY}}^3=\lambda \Omega\wedge\bar{\Omega},
\end{equation*}
and the infinity of $\omega_{\operatorname{CY}}$ models on the ALG ansatz with angle $\theta=\frac{2\pi }{k}$.

\end{theorem}

\begin{remark}


(i) The existence of a crepant resolution for the singularities at $f^{-1}(0)$ follows from the assumption $\tau^*\Omega_A=\zeta_k \Omega_A$ and Roan's general theorem \cite{roan1996minimal},  which asserts that if $G$ is a finite subgroup of $SL(3,\mathbb{C})$, then the quotient singularity $\mathbb{C}^3/G$ admits a crepant resolution. However, this result does not extend directly to higher dimensions without modification.

(ii) The theorem can be generalized by replacing the abelian surface with an arbitrary abelian variety, provided we can perform a crepant resolution at the singularities at $t=0$. The proof remains unchanged. However, in higher dimensions, a natural description of such resolutions is lacking.

(iii) Although the theorem builds on substantial prior work and its proof is not overly complex, it is listed separately because it offers valuable insights into constructing global models with negative canonical bundles.

\end{remark}

To maintain consistency with our approach, we employ the semi-flat metric introduced in Theorem \ref{semi-flat 1} as our model at infinity. However, upon closer examination through direct computation, it becomes evident that in the context of the quotient model, the semi-flat metric aligns perfectly with the flat metric inherited from $\mathbb{C}_t\times A$. This alignment facilitates a Kummer-type construction in this specific scenario.

More precisely, following a crepant resolution, we identify a metric defined in the vicinity of the exceptional divisors and expressed via a potential function. This metric extends the concept of the Eguchi-Hanson metric, which is traditionally defined on the blow-up of an $A_2$-singularity. Utilizing this metric for gluing operations offers two significant advantages:
\begin{itemize}

\item Proximity to Flatness: The metric closely resembles the flat metric, ensuring that the resulting glued metric remains positive definite post-gluing.

\item Enhanced Decay Estimates: Upon solving the non-compact Monge-Ampère equation, the final Calabi-Yau metric $\omega_{\operatorname{CY}}$ exhibits a remarkable proximity to our initial ansatz. Consequently, the decay estimates significantly surpass those obtained through the polynomial decay derived from the Tian-Yau-Hein framework, providing a more favorable outcome.

\end{itemize}

Referring to Kenji Ueno's paper (\cite{ueno1971fiber},\cite{ueno1972fiber}), we list the possible constructions such that there exists crepant resolution at $t=0$ (we omit some cases when the actions are the same while the lattices of abelian surfaces are different, because the calculating processes are almost the same):
\begin{itemize}

\item \textbf{Case 1}: $\mathbb{Z}_2\times (\mathbb{P}^1_{[t:s]}\times E_1\times E_2)\to \mathbb{P}^1_{[t:s]}\times E_1\times E_2,\; (-1,(t,w_1,w_2))\mapsto ((- t,-w_1, w_2))$, where $E_i$ is an arbitrary elliptic curve. Around $t=0$ and $t=\infty$, this corresponds to $\mathrm{II}$\ding{172}(a) type;

\item \textbf{Case 2}: $\mathbb{Z}_3\times (\mathbb{P}^1_{[t:s]}\times E_1\times E_2)\to \mathbb{P}^1_{[t:s]}\times E_1\times E_2,\; (\zeta_3,(t,w_1,w_2))\mapsto ((\zeta_3 t,\zeta_3^2 w_1, w_2))$, where $E_1$ is an elliptic curve corresponding to a hexagonal lattice, $E_2$ is an arbitrary elliptic curve. Around $t=0$, this corresponds to $\mathrm{III}$\ding{173}(b) type; around $t=\infty$, this corresponds to $\mathrm{III}$\ding{173}(a) type;

\item \textbf{Case 3}: $\mathbb{Z}_4\times (\mathbb{P}^1_{[t:s]}\times E_1\times E_2)\to \mathbb{P}^1_{[t:s]}\times E_1\times E_2,\; (i,(t,w_1,w_2))\mapsto ((i t,-i w_1, w_2))$, where $E_1$ is an elliptic curve corresponding to a square lattice, $E_2$ is an arbitrary elliptic curve. Around $t=0$, this corresponds to $\mathrm{III}$\ding{176}(a) type; around $t=\infty$, this corresponds to $\mathrm{III}$\ding{176}(b) type;

\item \textbf{Case 4}: $\mathbb{Z}_6\times (\mathbb{P}^1_{[t:s]}\times E_1\times E_2)\to \mathbb{P}^1_{[t:s]}\times E\times E,\; (\zeta_6,(t,w_1,w_2))\mapsto ((\zeta_6 t,\zeta_6^2 w_1, w_2))$, where $E_1$ is an elliptic curve corresponding to a hexagonal lattice, $E_2$ is an arbitrary elliptic curve. Around $t=0$, this corresponds to $\mathrm{III}$\ding{175}(a) type; around $t=\infty$, this corresponds to $\mathrm{III}$\ding{175}(b) type in Ueno's classification;

\item \textbf{Case 5}: $\mathbb{Z}_3\times (\mathbb{P}^1_{[t:s]}\times E\times E)\to \mathbb{P}^1_{[t:s]}\times E\times E,\; (\zeta_3,(t,w_1,w_2))\mapsto ((\zeta_3 t,\zeta_3 w_1,\zeta_3 w_2))$, where $E$ is an elliptic curve corresponding to a hexagonal lattice. Around $t=0$, this corresponds to $\mathrm{II}$\ding{174}(a) type; around $t=\infty$, this corresponds to $\mathrm{II}$\ding{174}(b) type;

\item \textbf{Case 6}: $\mathbb{Z}_4\times (\mathbb{P}^1_{[t:s]}\times E_1\times E_2)\to \mathbb{P}^1_{[t:s]}\times E_1\times E_2,\; (i,(t,w_1,w_2))\mapsto ((i t,i w_1, -w_2))$, where $E_1$ is an elliptic curve corresponding to a square lattice, $E_2$ is an arbitrary elliptic curve. Around $t=0$, this corresponds to $\mathrm{III}$\ding{177}(a) type; around $t=\infty$, this corresponds to $\mathrm{III}$\ding{177}(b) type;

\item \textbf{Case 7}: $\mathbb{Z}_6\times (\mathbb{P}^1_{[t:s]}\times E_1\times E_2)\to \mathbb{P}^1_{[t:s]}\times E_1\times E_2,\; (\zeta_6,(t,w_1,w_2))\mapsto ((\zeta_6 t,\zeta_3 w_1, -w_2))$, where $E_1$ is an elliptic curve corresponding to a hexagonal lattice, $E_2$ is an arbitrary elliptic curve. Around $t=0$, this corresponds to $\mathrm{III}$\ding{174}(a) type; around $t=\infty$, this corresponds to $\mathrm{III}$\ding{174}(b) type;

\item \textbf{Case 8}: $\mathbb{Z}_6\times (\mathbb{P}^1_{[t:s]}\times E\times E)\to \mathbb{P}^1_{[t:s]}\times E\times E,\; (\zeta_6,(t,w_1,w_2))\mapsto ((\zeta_6 t,\zeta_6 w_1, \zeta_3^2w_2))$, where $E$ is an elliptic curve corresponding to a hexagonal lattice. Unfortunately, this type of local model is not listed in Ueno's paper.

\item \textbf{Case 9}:  $\mathbb{Z}_{12}\times (\mathbb{P}^1_{[t:s]}\times E_1\times E_2)\to \mathbb{P}^1_{[t:s]}\times E_1\times E_2,\; (\zeta_{12},(t,w_1,w_2))\mapsto ((\zeta_{12} t,\zeta_6 w_1, -iw_2))$, where $E_1$ is an elliptic curve corresponding to a hexagonal lattice, $E_2$ is an elliptic curve corresponding to a square lattice. The singular fibers at $t=0,\infty$ belongs to class $\mathrm{IV}$\ding{178}.


\item \textbf{Case 10}: $\mathbb{Z}_{12}\times (\mathbb{P}^1_{[t:s]}\times E_1\times E_2)\to \mathbb{P}^1_{[t:s]}\times E_1\times E_2,\; (\zeta_{12},(t,w_1,w_2))\mapsto ((\zeta_{12} t,i w_1, \zeta_3^2 w_2))$, where $E_1$ is an elliptic curve corresponding to a square lattice, $E_2$ is an elliptic curve corresponding to a hexagonal lattice. The singular fibers at $t=0,\infty$ belongs to class $\mathrm{IV}$\ding{177}.


\item \textbf{Case 11}: $\mathbb{Z}_{5}\times (\mathbb{P}^1_{[t:s]}\times A)\to \mathbb{P}^1_{[t:s]}\times A,\; (\zeta_5,(t,w_1,w_2))\mapsto ((\zeta_5 t,\zeta_5 w_1, \zeta_5^3 w_2))$, where $A=\mathbb{C}^2/\Lambda$, where $\Lambda=\langle (1,1),(\zeta_5,\zeta_5^2),(\zeta_5^2,\zeta_5^4),(\zeta_5^3,\zeta_5)\rangle$ is an abelian surface with automorphism group $\mathbb{Z}/10\mathbb{Z}$. The singular fibers at $t=0,\infty$ belongs to class $\mathrm{IV}$\ding{177}.


\item \textbf{Case 12}: $\mathbb{Z}_6\times (\mathbb{P}^1_{[t:s]}\times E\times E)\to \mathbb{P}^1_{[t:s]}\times E\times E,\; (\zeta_6,(t,w_1,w_2))\mapsto ((\zeta_6 t,\zeta_3^2 w_2, \zeta_3^2 w_1))$, where $E$ is an elliptic curve corresponding to a hexagonal lattice. The singular fibers at $t=0,\infty$ belongs to class $\mathrm{IV}$\ding{173}.


\item \textbf{Case 13}: $\mathbb{Z}_6\times (\mathbb{P}^1_{[t:s]}\times E\times E)\to \mathbb{P}^1_{[t:s]}\times E\times E,\; (\zeta_6,(t,w_1,w_2))\mapsto ((\zeta_6 t,\zeta_6 w_2, \zeta_6 w_1))$, where $E$ is an elliptic curve corresponding to a hexagonal lattice. The singular fibers at $t=0,\infty$ belongs to class $\mathrm{IV}$\ding{173}.


\end{itemize}

\begin{remark}

First four cases are trivial because after a minimal resolution, the singular fibers in these cases are found to be the product of Kodaira-type singular fibers with a smooth elliptic curve, and we only need to take the product metric of Hein's constructions and the flat metric on elliptic curve.

\end{remark}

Now we give a proof to (1) and (2) of Theorem \ref{main 1}.

\begin{proof}

For the case $1\sim 4$, everything is trivial. For the case $5\sim 13$, the proofs are very similar. We only discuss case 5 in detail. That is we have to prove 
\begin{equation}
	K_X=-\frac23[F],
\end{equation}where $F=f^{*}(\infty)$ is the scheme theoretical fiber.

On $\mathbb{C}_t\times E\times E$, we have natural holomorphic volume form $dt\wedge dw_1\wedge dw_2$. This volume form is invariant under the $\mathbb{Z}_3$-action and hence descends naturally to $(\mathbb{P}^1_{[t:s]}\times E_{w_1}\times E_{w_2})/\mathbb{Z}_3$. By the basic property of crepant resolution, the descended volume form can be lifted to a volume form on the resolution space $X$. We just need to calculate the order of $\Omega$ along $F=f^{*}(\infty)$.

Locally the singularities at $t=\infty$ are modeled on $\mathbb{C}^3/\mathbb{Z}_3$, where $\mathbb{Z}_3$ acts on $\mathbb{C}^3$ by $\zeta\cdot (s,w_1,w_2)=(\zeta^2 s,\zeta w_1,\zeta w_2)$. Hence 
\begin{align*}
\mathbb{C}^3/\mathbb{Z}_3&=\operatorname{Spec}\mathbb{C}[s,w_1,w_2]^{\mathbb{Z}_3}=\operatorname{Spec}\mathbb{C}[s^3,w_1^3,w_2^3,sw_1,sw_2,w_1^2w_2,w_1w_2^2]=V(I),
\end{align*}where $I$ corresponds to the ideal of $\mathbb{C}[x_1,...,x_7]$ generated by the relations given by $x_1=s^3,...,x_7=w_1w_2^2$. Note $V(I)\to \Delta,\; (x_1,...,x_7)\mapsto x_1$ gives out the local fibration structure of $X\to \mathbb{P}^1$ around $s=0$. From this, the scheme theoretical fiber $f^{*}(\infty)$ at $t=\infty$ is not reduced. Actually, it's an irreducible smooth fiber with multiplicity $3$.

Now take $x_2,x_3,x_4$ to be the local coordinate of $V$. Suppose $\Omega=\pi^*(hdx_2\wedge dx_3\wedge dx_4)$ for some meromorphic function $h$ on $V$, that is 
\begin{align*}
\frac{1}{s^2}ds\wedge dw_1\wedge dw_2=\pi^*(hdx_2\wedge dx_3\wedge dx_4)=hd(w_1^3)\wedge d(w_2^3)\wedge d(sw_1)=9hw_1^3w_2^2ds\wedge dw_1 \wedge dw_2,
\end{align*}hence $h=\frac{1}{9s^2w_1^3w_2^2}=\frac{1}{9x_4^2x_7}$. Note $x_4=0$ refers to the reduced scheme $F_{\operatorname{red}}$ of $F$. Hence $K_X=-\frac23 [F]$.
\end{proof}

Next we compute the ansatz given by Theorem \ref{semi-flat 1}.

\begin{lemma}\label{flatness}
For the isotrivial case, the ansatz given in Theorem \ref{semi-flat 1} coincides with the flat metric descended from $\mathbb{C}_t\times A$.
\end{lemma}

\begin{proof}

We only explain \textbf{Case 13} in detail: Consider a neighborhood $U$ of $F$ in $X$. The pullback $U'$ of $U|_{\Delta^*}$ under $z=s^6$ can be written as $(\Delta^*_s\times \mathbb{C}_{w_1}\times \mathbb{C}_{w_2})/(\mathbb{Z}+\zeta_3\mathbb{Z}+\mathbb{Z}+\zeta_3\mathbb{Z})$ and thus the central fiber at $s=0$ is a smooth torus $\mathbb{C}^2/(\mathbb{Z}+\zeta_3\mathbb{Z}+\mathbb{Z}+\zeta_3\mathbb{Z})$. The deck transformation of the monodromy on $U'$ is generated by 
\begin{equation*}
\mathscr{A}:U'\to U',\; (s,w_1,w_2)\mapsto (\zeta_6^5s,\zeta_6 w_2,\zeta_6 w_1).
\end{equation*}Now the fiberwise linear map
\begin{equation*}
\Phi:\Delta^*\times \mathbb{C}\times \mathbb{C}\to \Delta^*\times \mathbb{C}\times \mathbb{C},\; (s,w_1,w_2)\mapsto (s^6,s(w_1+w_2),s^4(w_1-w_2))=(z,v_1,v_2)
\end{equation*}factors through the monodromy action of $A$. Hence we have an isomorphism of fibrations
\begin{align*}
\tilde{\Phi}:U|_{\Delta^*}\to (\Delta_z^*\times \mathbb{C}_{v_1}\times \mathbb{C}_{v_2})/(\mathbb{Z}\tau_1+\mathbb{Z}\tau_2+\mathbb{Z}\tau_3+\mathbb{Z}\tau_4)
\end{align*}with multi-valued generators
\begin{align*}
&\tau_1(z)=\operatorname{pr}_2\tilde{\Phi}\left(z^{\frac13},1,0\right)=z^{\frac16};\\
&\tau_2(z)=\operatorname{pr}_2\tilde{\Phi}\left(z^{\frac13},\zeta_3,0\right)=\zeta_3z^{\frac16};\\
&\tau_3(z)=\operatorname{pr}_3\tilde{\Phi}\left(z^{\frac13},0,1\right)=z^{\frac23};\\
&\tau_4(z)=\operatorname{pr}_3\tilde{\Phi}\left(z^{\frac13},0,\zeta_3\right)=\zeta_3z^{\frac23}.
\end{align*}Hence
\begin{equation*}
T=\begin{pmatrix}
z^{\frac16}&\zeta_3z^{\frac16}&0&0\\
0&0&z^{\frac23}&\zeta_3z^{\frac23}
\end{pmatrix}
\end{equation*}

Now we need to determine the multiplicity $N$ of $\tilde{\Phi}(dz\wedge dv_1\wedge dv_2)$ along the central fiber $F$. Note that we have the following commutative diagram of fibrations:
\begin{displaymath}
\xymatrix{U\ar[d]\ar[rd]^{\tilde{\Phi}}&\\
V\ar[r]& (\Delta_z^*\times \mathbb{C}^2_{v_1,v_2})/\Lambda(z)},    
\end{displaymath}where the lattice $\Lambda(z)=\mathbb{Z}\tau_1+\mathbb{Z}\tau_2+\mathbb{Z}\tau_3+\mathbb{Z}\tau_4$.

Note that 
\begin{equation*}
\mathbb{C}^3/\mathbb{Z}^6=\operatorname{Spec}\mathbb{C}[s^6,sw_1+sw_2,w_1^3w_2^3,w_1^6+w_2^6,s^2w_1w_2,s^4w_1-s^4w_2,\cdots],
\end{equation*}Take $x_1=sw_1+sw_2$, $x_2=w_1^3w_2^3$, $x_3=w_1^6+w_2^6$ to be the coordinate.

Direct computation shows that
\begin{align*}
\tilde{\Phi}^*(dz\wedge dv_1\wedge dv_2)&=d(s^6)\wedge d(sw_1+sw_2)\wedge d(s^4w_1-s^4w_2)\\
&=-12s^{10}ds\wedge dw_1\wedge dw_2\\
&=x_1^{10}m(x_2,x_3)dx_1\wedge dx_2\wedge dx_3,
\end{align*}where $m(x_2,x_3)$ is rational function. Hence the multiplicity of $\tilde{\Phi}^*(dz\wedge dv)$ along $F_{\operatorname{red}}$ is $10$. 

Assume $\Omega=\frac{1}{s^2}ds\wedge dw_1\wedge dw_2=hdx_1\wedge dx_2\wedge dx_3$, then $h=\frac{n(w_1,w_2)}{s^2}ds\wedge dw_1\wedge dw_2$, where $n(x_2,x_3)$ is rational function. Therefore $\Omega$ has a pole of order $2$ along $F_{\operatorname{red}}$, that is $K_X=-\frac{2}{6}[F]$.

 Thus we should take $g(z)=-\frac{1}{12z^2}$ so that the pullback of $g(z)dz\wedge dv_1\wedge dv_2$ is just the holomorphic volume $\Omega=\frac{1}{s^2}ds\wedge dw_1\wedge dw_2$ descended to $X$. 


Note that 
\begin{equation*}
\omega=\frac{i\varepsilon}{\sqrt{3}}(dw_1\wedge d\bar{w}_1+dw_2\wedge d\bar{w}_2)
\end{equation*}can be descended to $U|_{\Delta^*}$ and become a constant polarization such that the volume of the fiber is $\varepsilon$. Change variable by
\begin{equation*}
s=z^{\frac16},\; w_1=\frac12\left(z^{-\frac16}v_1+z^{-\frac23}v_2\right),\; w_2=\frac12\left(z^{-\frac16}v_1-z^{-\frac23}v_2\right),
\end{equation*}then we have
\begin{equation*}
\frac{i\varepsilon}{\sqrt{3}}\left(dw_1\wedge d\bar{w}_1+dw_2\wedge d\bar{w}_2\right)|_{\hat{F}}=\frac{i\varepsilon}{2\sqrt{3}}|z|^{-\frac13}dv_1\wedge d\bar{v}_1+\frac{i\varepsilon}{2\sqrt{3}}|z|^{-\frac43}dv_2\wedge d\bar{v}_2,
\end{equation*}where $\hat{F}$ is an arbitrary smooth fiber in $U|_{\Delta^*}$. From this we get
\begin{equation*}
H(\varepsilon)_{11}=\frac{\varepsilon}{2\sqrt{3}}|z|^{-\frac13},\; H(\varepsilon)_{22}=\frac{\varepsilon}{2\sqrt{3}}|z|^{-\frac43},\; H(\varepsilon)_{12}=H(\varepsilon)_{21}=0,
\end{equation*}which gives $\det(H(\varepsilon))^{-1}=\frac{12}{\varepsilon^2}|z|^{\frac53}$. 

For the Christoffel symbols, we compute as follows
\begin{align*}
\Gamma(z,v)&=\begin{pmatrix}
\Gamma^1(z,v_1,v_2)\\
\Gamma^2(z,v_1,v_2)
\end{pmatrix} =\frac{\partial T}{\partial z}\binom{T}{\bar{T}}^{-1}\binom{v}{\bar{v}}\\
&=\begin{pmatrix}
&\frac{1}{\operatorname{Im}\left(\bar{\tau}_1 \tau_2\right)}\left(\operatorname{Im}\left(\bar{\tau}_1 v_1\right) \tau_2^{\prime}-\operatorname{Im}\left(\bar{\tau}_2 v_1\right) \tau_1^{\prime}\right)\\
&\frac{1}{\operatorname{Im}\left(\bar{\tau}_3 \tau_4\right)}\left(\operatorname{Im}\left(\bar{\tau}_3 v_2\right) \tau_2^{\prime}-\operatorname{Im}\left(\bar{\tau}_4 v_2\right) \tau_3^{\prime}\right)
\end{pmatrix}\\
&=\begin{pmatrix}
\frac{v_1}{6z}\\
\frac{2v_2}{3z}
\end{pmatrix}.
\end{align*}

We put all the data into formula \eqref{semi-flat 2}, then
\begin{align*}
\omega_{\mathrm{sf}, \varepsilon}&=\frac{i|z|^{\frac53}}{12\varepsilon^2}\frac{dz\wedge d\bar{z}}{|z|^4}\\
&+\frac{i\varepsilon}{2\sqrt{3}|z|^{\frac13}}\left(d v_1-\frac16 \frac{v_1}{z} d z\right) \wedge\left(d \bar{v}_1-\frac16 \frac{\bar{v}_1}{\bar{z}} d \bar{z}\right)\\
&+\frac{i\varepsilon}{2\sqrt{3}|z|^{\frac43}}\left(d v_2-\frac23 \frac{v_2}{z} d z\right) \wedge\left(d \bar{v}_2-\frac23 \frac{\bar{v}_2}{\bar{z}} d \bar{z}\right)
\end{align*}

Change coordinates on $\Delta\backslash[0,1)$ by setting $z=\left(\frac{\alpha_0}{\alpha}\right)^{6}$ and $v_1=\left(\frac{\alpha_0}{\alpha}\right)^{1}(\beta_1+\beta_2)$, $v_2=\left(\frac{\alpha_0}{\alpha}\right)^{4}(\beta_1-\beta_2)$. Then $\alpha$ ranges over the sector $|\alpha|>\alpha_0$, $0<\arg \alpha<\frac{2\pi }{6}$, and $\beta_i$ move in the torus $\mathbb{C}/\mathbb{Z}+\zeta_3\mathbb{Z}$. In the new coordinate, with a carefully chosen constant $\alpha_0=-\frac{\sqrt{6}}{\varepsilon}$, the above ansatz looks like
\begin{equation}
\omega_{\operatorname{sf},\varepsilon}=\frac{i}{2}\left\{ d\alpha\wedge d\bar{\alpha}+\frac{\varepsilon}{\sqrt{3}}(d\beta_1\wedge d\bar{\beta}_1+d\beta_2\wedge d\bar{\beta}_2)\right\}.
\end{equation}This is automatically a flat metric. So the semi-flat ansatz is just a complex $2$-dimensional torus fibration over a cone whose angle is $\frac{1}{3}$. It's clear that the semi-flat ansatz coincides with the flat metric descended from $\mathbb{C}_t\times A$.

\end{proof}

\begin{lemma}

There exists a k{\"a}hler metric $\omega_{0,\varepsilon}$ on $X\backslash F$ which coincides with the pullback of $\omega_{\operatorname{sf},\varepsilon}$ outside a neighborhood of the $t=0$ fiber.

\end{lemma}

\begin{proof}

For simplicity, we only treat case 5 and case 6 in detail.

\textbf{Case 5:} In this case, there are only isolated singularities. Here we have $9$ singularities of $A_2$-type on the $t=0$ fiber. After crepant resolution, we get $9$ exceptional divisors. The local neighborhoods of the $9$ exceptional divisors are modeled on $\mathcal{O}_{\mathbb{P}^2}(-3)$, and the exceptional divisors correspond to the the zero section. From chapter 4 of \cite{lye2019stable}, we know that by solving an ODE, we can construct a generalized Eguchi-Hanson metric on the Crepant resolution of $\mathbb{C}^3/\mathbb{Z}^3$ whose action is given by $(\zeta, (t,w_1,w_2))\mapsto (\zeta t,\zeta w_1,\zeta w_2)$. Since Eguchi-Hanson metric is very close to the euclidean metric, we can imitate chapter 5 of \cite{lye2019stable} by using a cutoff function to glue the $9$ Eguchi-Hanson metrics to the descended flat metric $\omega_{\mathrm{sf},\varepsilon}$ to get a K{\"a}hler metric $\omega_{0,\varepsilon}$. Note that $\omega_{0,\varepsilon}$ only changes in a neighborhood of the $9$ exceptional divisors and remains unchanged near $t=\infty$.

To be more precise, we have natural identifications $(\mathbb{C}^3\backslash\{0\})/\mathbb{Z}_3\cong \mathcal{O}_{\mathbb{P}^2}(-3)\backslash \mathbb{P}^2$ and $\operatorname{Bl}_0(\mathbb{C}^3\backslash \mathbb{Z}_3)\cong \mathcal{O}_{\mathbb{P}^2}(-3)\cong \operatorname{Bl}_0\mathbb{C}^3/\mathbb{Z}_3$. The Eguchi-Hanson metric on $(\mathbb{C}^3\backslash\{0\})/\mathbb{Z}_3$ is given by 
\begin{equation*}
g_{i\bar{j}}=\sqrt[3]{1+\frac{a^3}{u^3}}\left(\delta_{ij}-\frac{a^3}{a^3+u^3}\frac{\bar{z}_iz_j}{u}\right),
\end{equation*}for some parameter $a>0$, $u=|z|^2$. Note that this metric is invariant under $\mathbb{Z}_3$-action, so it's well-defined on $(\mathbb{C}^3\backslash\{0\})/\mathbb{Z}_3$. There is a global K{\"a}hler potential for this metric: 
\begin{equation*}
f_a(u)=\sqrt[3]{a^3+u^3}+\frac{a}{3}\sum\limits_{j=0}^2\zeta^j\log\left(\sqrt[3]{1+\frac{u^3}{a^3}}-\zeta^j\right).
\end{equation*}This metric can be extended to a metric on $\operatorname{Bl}_0(\mathbb{C}^3/ \mathbb{Z}_3)\cong \mathcal{O}_{\mathbb{P}^2}(-3)$. Let $\chi:[0,\infty)\to \mathbb{R}$ be a smooth cutoff function such that $\chi(u)=1$, $\forall u\leq 1$, $\chi(u)=0$, $\forall u\geq 1+\delta$, $\chi^{(k)}(u)=0$ for $u=1,1+\delta$, $\forall k\geq 1$. Then $\Phi_a(u)=u+\chi(u)(f_a(u)-u)$ defines a potential on $\operatorname{Bl}_0(\mathbb{C}^3/ \mathbb{Z}_3)$. The k{\"a}hler metric $g_a$ associated to this K{\"a}hler potential satisfies $|g_a-g_{\operatorname{euc}}|_{C^k}\leq C_ka^2$, where $C_k$ is a constant independent of $a$. So if we take $a$ sufficiently small, $g_a$ is indeed a K{\"a}hler metric. By definition of $f_a$, we know in a neighborhood of $0$, the metric is Eguchi-Hanson, while outside the neighborhood the metric is the flat one. We can carry out this process around the nine singularities in Case 5.

For general cases, the singularities may not be isolated.

\textbf{Case 6:} First, we determine fixed locus at $t=0$. Note $\mathbb{Z}_4$-action is generated by $\rho:(t,w_1,w_2)\mapsto (it,iw_1,-w_2)$. So $\langle \rho^2\rangle$ form a $\mathbb{Z}_2$-subgroup. 

For the $\mathbb{Z}_2$ action, the fixed loci are given by
\begin{align*}
F_1=\left\{(0,0,z_1)\right\},\; F_2=\left\{\left(0,\frac12+\frac i2,z_2\right)\right\},\;F_3= \left\{\left(0,\frac12,z_3\right)\right\},\; F_4=\left\{\left(0,\frac i2,z_4\right)\right\},
\end{align*}where $z_1,z_2,z_3,z_4$ ranges in $E_2$. Note that the normal bundles of $F_i$ are all trivial. Under the $\mathbb{Z}_4$-action, $F_1$ and $F_2$ turn into rational curves $\tilde{F}_1$, $\tilde{F}_2$, while $F_3$ and $F_4$ correspond to the same torus $\tilde{F}_3$. 

For the $\mathbb{Z}_4$ action, there are eight  fixed points
\begin{align*}
&P_1=(0,0,0),\; P_2=\left(0,0,\frac12\right),\; P_3=\left(0,0,\frac i2\right),\; P_4=\left(0,0,\frac12+\frac i2\right);\\
&P_5=\left(0,\frac12+\frac i2,0\right),\; P_6=\left(0,\frac12+\frac i2,\frac12\right),\; P_7=\left(0,\frac12+\frac i2,\frac i2\right),\; P_8=\left(0,\frac12+\frac i2,\frac12+\frac i2\right).
\end{align*}The first four points belong to $F_1$, while the latter four points belong to $F_2$. In a small neighborhood $U_i$ of each $P_i$, the $Y=(\mathbb{C}\times E_1\times E_2)/\mathbb{Z}_4$ is locally isomorphic to $\mathbb{C}^3/\mathbb{Z}_4$. Now we take a crepant resolution $f:X\to Y$. So $f^{-1}(U_i)$ is isomorphic to an open subset $V_i$ of a crepant resolution of $\mathbb{C}^3/\mathbb{Z}_4$. Choose a small open neighborhood $U$ of $\tilde{F}_1$, so that $f^{-1}\left(U\backslash \bigcup\limits_{i=1}^4U_i\right)=f^{-1}(U)\backslash \bigcup\limits_{i=1}^4 V_i$ has product structure, i.e. we have the following commutative diagram
\begin{displaymath}
\xymatrix{ F_1\times \widetilde{(\mathbb{C}^2/\mathbb{Z}_2) } \ar[d] &   \widetilde{W}  \ar@{_{(}->}[l]  \ar[r]\ar[d]& f^{-1}(U)\backslash \bigcup\limits_{i=1}^4 V_i \ar[d]\\
F_1\times (\mathbb{C}^2/\mathbb{Z}_2) &  W   \ar@{_{(}->}[l]    \ar[r]&  U\backslash \bigcup\limits_{i=1}^4 U_i   },
\end{displaymath}where $W$ is a suitable open set in $F_1\times (\mathbb{C}^2/\mathbb{Z}_2)$ such that the right block is a cartesian diagram, $\widetilde{(\mathbb{C}^2/\mathbb{Z}_2)}\cong \mathcal{O}_{\mathbb{P}^1}(-2)$ is the crepant resolution of $\mathbb{C}^2/\mathbb{Z}_2$. $W$ can be viewed as the quotient of the normal bundle of $F_1$ under the $\mathbb{Z}_2$-action. So we can descend the flat k{\"a}hler metric $\omega_0$ on $\mathbb{C}\times E_1\times E_2$ to $W\backslash (F_1\times \{0\})$. Take the Eguchi-Hanson metric $\omega_{\operatorname{EH}}$ on $\widetilde{(\mathbb{C}^2/\mathbb{Z}_2) }$ and glue $\omega_{\operatorname{EH}}\oplus h_{F_1}$ with $\omega_0$ by using cutoff function as shown in Case 5 to a K{\"a}hler metric $\omega_1$, where $h_{F_1}$ is the flat metric on torus. Note that the above argument is based on the fact that the normal bundle of $F_1$ is trivial. Since the Eguchi-Hanson metric and flat metric are both invariant under $\mathbb{Z}_4$-action, so we can descend $\omega_1$ to $f^{-1}(U)\backslash \bigcup\limits_{i=1}^4 V_i$. Now around $f^{-1}(P_i)$, we can apply the proof of theorem 4.7 in \cite{joyce2001quasi} to get a K{\"a}hler metric in $V_i$ which coincides with $\omega_1$ outside $V_i$. 

The metric around $\tilde{F}_2$ can be handled identically as $\tilde{F}_1$, while the metric around $\tilde{F}_3$ can simply be taken as $\omega_1$. Note that these three metrics coincide with flat metric outside a suitable neighborhood, so they can be glued together as a k{\"a}hler metric on the crepant resolution $X$.

A similar argument applies to other cases because all the dimension $1$ fixed loci have trivial normal bundles.


\end{proof}

We can then adjust the resulting global ansatz metric to meet the requirements of the Tian-Yau-Hein method: since $\omega_{0,\varepsilon}$ is flat at infinity, Lemma \ref{QuasiAtlas} ensures the existence of quasi-coordinate charts. Lemma \ref{flatness} shows that $\omega_{0,\varepsilon}$ satisfies the SOB(2) condition. Using the adjustment procedure from Section \ref{perturbation}, we can modify $\omega_{0,\varepsilon}$ to obtain a new ansatz metric $\omega'_{0,\varepsilon}$ which is flat at infinity and satisfies the integrability condition $\int_M (e^f - 1)(\omega'_{0,\varepsilon})^3 = 0$. Therefore, we are able to apply the Tian-Yau-Hein method (Theorem \ref{TYH package}) to perturb the ansatz to a Calabi-Yau metric, which concludes the proof of Theorem \ref{main 1}.

From Proposition 1.6 in \cite{wang2022ricci}, we can get a good decay estimate in our isotrivial construction.

\begin{proposition}[Schwartz Decay]\label{schwartz}

Let $u$ be a smooth solution of the complex Monge-Ampère equation $(\omega'_{0,\varepsilon}+i\partial\bar{\partial}u)^3=C{\omega'_{0,\varepsilon}}^{3}$ given by Tian-Yau-Hein package, then for any $\beta\in \mathbb{R}^+$ and $k\in\mathbb{N}$, we have
\begin{equation*}
\left|\nabla^k\partial\bar{\partial}u\right|_{\omega'_{0,\varepsilon}}=O\left(\frac{1}{r^{\beta+k}}\right).
\end{equation*}This kind of decay is called Schwartz decay.

\end{proposition}

To conclude this section, we discuss those cases not admitting crepant resolutions. According to \cite{crauder1994minimal}, in Ueno's list, there are still three singular fibers that have smooth relatively minimal models, i.e $\mathrm{II}$\ding{173}(b), $\mathrm{II}$\ding{175}(b), $\mathrm{III}$\ding{172}(b). In these three cases, after the canonical resolution constructed by Ueno, we have to do contraction and simple flip so that the resulting model is relatively minimal. Nevertheless, we can still descend a holomorphic volume form onto these spaces. This is somewhat similar to (2) in Theorem \ref{main 1}.

We discuss the case $\mathrm{II}$\ding{173}(b) in detail.

Consider the action $\mathbb{Z}_4\times (\mathbb{P}^1_{[t:s]}\times E\times E)\to \mathbb{P}^1_{[t:s]}\times E\times E,\; (i,(t,w_1,w_2))\mapsto (i t,i w_1,i w_2)$, where $E$ is an elliptic curve corresponding to a square lattice. Around $t=0$, this corresponds to $\mathrm{II}$\ding{173}(b) type; around $t=\infty$, this corresponds to $\mathrm{II}$\ding{173}(a) type. 

Now consider the rational form $\Omega=tdt\wedge dw_1\wedge dw_2$ on $\mathbb{P}_{[t:s]}\times E\times E$. This form is invariant under the $\mathbb{Z}_4$-action, so it can be descended to the quotient space. From the proof of theorem 1.1 in \cite{crauder1994minimal}, we know the descended form can be lifted to the relatively minimal model $X$ and has no zero or pole except at $t=\infty$. We still denote this holomorphic volume form on $X$ by $\Omega$. By a similar calculation in the previous paragraph, the fiber $F$ at $t=\infty$ is of multiplicity 4, and $\Omega$ has a pole of order $2$ along $F_{\mathrm{red}}$. Then we can calculate the ansatz in the coordinate $(z,v_1,v_2)=(s^4,s^3v_1,s^3v_2)$, which will be a flat cone metric with angle $\frac14$. The essential difficulty is that we could not carry out the gluing process as in the crepant resolution case. 
Nevertheless, we still expect there exist such ALG K{\"a}hler Ricci flat metrics. 

In the next section, we turn to some non-isotrivial construction, where such difficulties arsing from gluing process could be avoided.

\section{Construction of non-isotrivial models}\label{non-isotrivial models}

In this section, we construct non-isotrivial models via fiber product.

\subsection{Elaboration on the main Theorem \ref{main 0}}\label{elaboration}

\begin{theorem}[Main theorem]\label{main 2}

Let $\pi_i:X_i\to \mathbb{P}^1$, $i=1,2$ be two rational elliptic fibrations. Specify one particular singular fiber $F_i$ on each $X_i$, $i=1,2$. Let the two singular fibers $F_1,F_2$ project to the same point on $\mathbb{P}^1$ while all the other singular fibers to different points. Take the fiber product $f:X=X_1\times_{\mathbb{P}^1}X_2\to \mathbb{P}^1$. Now $X$ is a singular variety with all the singular loci contained in the singular fiber $F=F_1\times F_2$. Let $M=X\backslash F$. We have:

(1) If the monodromy of both $F_1,F_2$ are finite, then
\begin{itemize}

\item (1a) In the punctured neighborhood of $F$, $M$ has a local quotient model as $(\Delta^*_s\times E_{1,w_1}\times E_{2,w_2})/\sim$, where $\sim$ is given by the deck transformation of the monodromy action and has the form
\begin{equation*}
\mathscr{A}:(s,w_1,w_2)\mapsto \left(\zeta_k s,\zeta_k^{\alpha}h_1(s)w_1,\zeta_k^{\beta}h_2(s)w_2\right),
\end{equation*}where $h_i(s)$ is determined by the monodromy action of the elliptic fibration and $1\leq \alpha,\beta<k$.

\item (1b) We can naturally compactify the fibration $M\to \mathbb{C}$ into a normal projective variety $\widetilde{X}$ fibered over $\mathbb{P}^1$ such that \begin{equation*}
K_{\widetilde{X}}=\frac{k-\alpha-\beta-1}{k}[\widetilde{F}],
\end{equation*}where $\widetilde{F}$ is an arbitrary fiber of $\widetilde{X}\to \mathbb{P}^1$.

\item (1c) When $\alpha+\beta>k$, there exists an ALG Ricci flat K{\"a}hler metric on $X\backslash F$ with angle $\theta=\frac{2(\alpha+\beta-k)\pi}{k}$.

\item (1d) When $\alpha+\beta=k$, there exists an ALH Ricci flat K{\"a}hler metric on $X\backslash F$.

\end{itemize}

(2) If $F_i$ is of $I_{b_i}^*$ type, then

(2a) We can naturally compactify the fibration $M\to \mathbb{C}$ into a normal projective variety $\widetilde{X}$ fibered over $\mathbb{P}^1$ such that \begin{equation*}
K_{\widetilde{X}}=-\frac12[\widetilde{F}],
\end{equation*}where $\widetilde{F}$ is an arbitrary fiber of $\widetilde{X}\to \mathbb{P}^1$.

(2b) There exists a complete Calabi-Yau metric $\omega_{\operatorname{CY}}$ on $M$ satisfying $|B(x_0,s)|\sim s^{\frac32}$ for $s>>1$ and the unique tangent cone at infinity is $\mathbb{R}^+$.

(3) If $F_1$ is of $I_{b}^*$ type and $F_2$ is of $\mathrm{II}^*$ or $\mathrm{III}^*$ or $\mathrm{IV}^*$ type, then 

(3a) We can naturally compactify the fibration $M\to \mathbb{C}$ into a normal projective variety $\widetilde{X}$ fibered over $\mathbb{P}^1$ such that \begin{equation*}
K_{\widetilde{X}}=-\frac12 [\widetilde{F}],\; -\frac12 [\widetilde{F}],\; -\frac13 [\widetilde{F}],
\end{equation*}corresponding to $\mathrm{II}^*$, $\mathrm{III}^*$ and $\mathrm{IV}^*$ type respectively, where $\widetilde{F}$ is an arbitrary fiber of $\widetilde{X}\to \mathbb{P}^1$.

(3b) There exists a complete Calabi-Yau metric $\omega_{\operatorname{CY}}$ on $M$ satisfying $|B(x_0,s)|\sim s^{2}$ for $s>>1$ and the unique tangent cone at infinity is a metric cone of angle $\frac{2\pi}{3}$, $\frac{\pi}{2}$ and $\frac{\pi}{3}$ corresponding to $\mathrm{II}^*$, $\mathrm{III}^*$ and $\mathrm{IV}^*$ type respectively.

\end{theorem}

\begin{remark}

(i) One may argue that the fiber product of two rational elliptic surfaces is already a Calabi-Yau variety by using the adjunction formula. However, this argument is valid only for generic cases. In our cases, the fiber products are not normal, so we cannot define the canonical divisor. This is caused by the multiplicities of irreducible components in the singular fibers (for example $\mathrm{II}^*$, $\mathrm{III}^*$, $\mathrm{IV}^*$, $\mathrm{I}_b^*$).


(ii) The Calabi-Yau metrics in (2) can be viewed as generalized $\mathrm{ALH}^*$-type metrics while the the Calabi-Yau metrics in (3) can be viewed as generalized $\mathrm{ALG}^*$-type metrics.

\end{remark}

\subsection{Review of local models on elliptic fibrations}

Since our construction of non-isotrivial ansatz relies heavily on the fiber product structure of two elliptic fibrations, we need to review local model of singular fiber of elliptics fibrations. We follow the argument in \cite{hein2010gravitational} and \cite{kodaira1963compact}.

Consider an elliptic fibration $\pi:U\to \Delta$ with only one singular fiber at $0$. Then $U|_{\Delta^*}$ is abstractly isomorphic to some $(\Delta_z^*\times \mathbb{C}_v)/(\mathbb{Z}\tau_1(z)+\mathbb{Z}\tau_2(z))$, where $\tau_i(z)$ are multi-valued functions, and $(z,v)$ forms a coordinate. By adjunction formula or Poincar\'{e} residue formula, we can always write a holomorphic volume form $\Omega$ on $U$ in the coordinate of $(z,v)$ as $\Omega=g(z)dz\wedge dv$. However, the coordinate $(z,v)$ is not defined on the singular fiber, so the expression $g(z)dz\wedge dv$ is in a priori defined on $U|_{\Delta^*}$ not on $U$. Nevertheless, by assumption, $\Omega$ is defined on the whole $U$. So when we choose a coordinate $(x,y)$ around a point on the singular fiber, we can compute $dz\wedge dw$ in the new coordinate $(x,y)$ and get something as $\pi^Nh(x,y)dx\wedge dy$, where $N$ is the multiplicity of $\Omega$ along the singular fiber $\pi^{-1}(0)$. Since the semi-flat ansatz is computed in the $(z,v)$ coordinate, it's essential to determine the multiplicity $N$. 

\noindent\textbf{Case 1: Finite monodromy}: We begin by recalling the local model around a singular fiber with finite monodromy in an elliptic fibration. Let $\pi:U\to \Delta_z$ be an elliptic fibration, such that $z=0$ is the only singular fiber whose monodromy is finite. 

We only deal with the singular fiber of type $\mathrm{IV}$; other cases can be treated similarly. According to Kodaira, the $j$-invariant has zero at $z=0$. That is $\mathcal{J}(0)=0$, and let $m:=\operatorname{mult}_0 \mathcal{J} \in \mathbb{N}$. From Kodaira \cite{kodaira1963compact}, we know that $m\equiv 2$ $(\operatorname{mod}3)$. When $\tau$ approaches $\zeta_3$, we can perform a coordinate change on $\Delta$, so that 
\begin{align*}
\tau(z)=\zeta_3 \frac{1-\zeta_3 z^{\frac{m}{3}}}{1-z^{\frac{m}{3}}}
\end{align*}
Traversing a counterclockwise loop around $z=0$ once induces a transformation on $\tau$ given by
\begin{align*}
\tau \mapsto A \tau=\frac{d \tau+b}{c \tau+a}
\end{align*}where
\begin{align*}
 A=\begin{pmatrix}
-1 & 1\\
-1 & 0
 \end{pmatrix}
\end{align*}

The pullback $U^{\prime}$ of $\left.U\right|_{\Delta^*}$ under $z=s^3$ can be written as $\left(\Delta_s^* \times \mathbb{C}_w\right) /\left(\mathbb{Z}+\mathbb{Z} \tau\left(s^6\right)\right)$ and thus extends over $s=0$ with central fiber $\mathbb{C} /\left(\mathbb{Z}+\mathbb{Z} \zeta_3\right)$. The deck action of the monodromy on $U^{\prime}$ is generated by
\begin{align*}
\mathscr{A}(s, w)=\left(\zeta_3 s, \zeta_3 \frac{1-s^{2 m}}{1-\zeta_3 s^{2 m}} w\right)
\end{align*}
Note that 
\begin{align*}
\Phi: \Delta^* \times \mathbb{C} \rightarrow \Delta^* \times \mathbb{C},(s, w) \mapsto\left(s^3,\left(1-s^{2 m}\right) s^2 w\right)
\end{align*}factors through the monodromy action $A$, thus inducing an isomorphism of elliptic fibrations
\begin{align*}
\Phi:U|_{\Delta^*}=U'/\sim_{\mathscr{A}}\longrightarrow (\Delta^*\times \mathbb{C})/(\mathbb{Z}\tau_1+\mathbb{Z}\tau_2),
\end{align*}where 
\begin{align*}
&\tau_1(z)=\operatorname{pr}_2\Phi\left(z^{\frac13},1\right)=(1-z^{\frac m3})z^{\frac23};\\  
&\tau_2(z)=\operatorname{pr}_2\Phi\left(z^{\frac13},\zeta_3\right)=\zeta_3(1-\zeta_3z^{\frac m3})z^{\frac23}.
\end{align*}
Moreover, $(z,v)=\left(s^3,\left(1-s^{2 m}\right) s^2 w\right)$ can be viewed as a coordinate on $U|_{\Delta^*}$. From Hein \cite{hein2010gravitational}, by analyzing the canonical resolution, we find that the multiplicity $N$ of the holomorphic volume form along the singular fiber is $N=1$.

For sake of readers' convenience, we list all the cases in the following table:

\begin{center}
\begin{tabular}{|c|c|c|c|c|}
	\hline $\pi^*(0)$ & Matrix $A$ &  Deck action $\mathscr{A}$ & Coordinate $(z,v)$& generators $\tau_1$, $\tau_2$ \\
	\hline $\mathrm{I}^*_0$    & $\begin{pmatrix}
		-1& 0\\
		0& -1
	\end{pmatrix}$ & $(s,w)\mapsto(-s,-w)$ & $(s^2,sw)$ & $z^{\frac12}$, $z^{\frac12}\tau(z)$  \\
	\hline $\mathrm{II}$ & $\begin{pmatrix}
		0& 1\\
		-1& 1
	\end{pmatrix}$ & $(s,w)\mapsto \left(\zeta_6s,\zeta_6\frac{1-s^{2m}}{1-\zeta_3s^{2m}}\right)w $ & $(s^6,(1-s^{2m})s^5w)$ &  $(1-z^{\frac m3})z^{\frac56}$, $\zeta_3(1-\zeta_3z^{\frac m3})z^{\frac56}$  \\
	\hline $\mathrm{II}^*$ & $\begin{pmatrix}
		1& -1\\
		1& 0
	\end{pmatrix}$ & $(s,w)\mapsto \left(\zeta_6s,\zeta_6^5\frac{1-s^{2m}}{1-\zeta_3s^{2m}}\right)w $ & $(s^6,(1-s^{2m})sw)$ & $(1-z^{\frac m3})z^{\frac16}$, $\zeta_3(1-\zeta_3z^{\frac m3})z^{\frac16}$   \\
	\hline $\mathrm{III}$ & $\begin{pmatrix}
		0& 1\\
		-1& 0
	\end{pmatrix}$ & $(s,w)\mapsto \left(is,i\frac{1-s^{2m}}{1-i s^{2m}}\right)w$ & $(s^4,(1-s^{2m})s^3w)$ &  $(1-z^{\frac m2})z^{\frac34}$, $i(1+z^{\frac m2})z^{\frac34}$  \\
	\hline $\mathrm{III}^*$ & $\begin{pmatrix}
		0& -1\\
		1& 0
	\end{pmatrix}$ & $(s,w)\mapsto \left(is,-i\frac{1-s^{2m}}{1-i s^{2m}}\right)w$ & $(s^4,(1-s^{2m})sw)$ & $(1-z^{\frac m2})z^{\frac14}$, $i(1+z^{\frac m2})z^{\frac14}$ \\
	\hline $\mathrm{IV}$ & $\begin{pmatrix}
		-1& 1\\
		-1& 0
	\end{pmatrix}$ & $(s,w)\mapsto \left(\zeta_3s,\zeta_3\frac{1-s^{2m}}{1-\zeta_3s^{2m}}\right)w $ & $(s^3,(1-s^{2m})s^2w)$ & $(1-z^{\frac m3})z^{\frac23}$, $\zeta_3(1-\zeta_3z^{\frac m3})z^{\frac23}$\\
	\hline $\mathrm{IV}^*$ & $\begin{pmatrix}
		0& -1\\
		1& -1
	\end{pmatrix}$ & $(s,w)\mapsto \left(\zeta_3s,\zeta_3^2\frac{1-s^{2m}}{1-\zeta_3s^{2m}}\right)w $ & $(s^3,(1-s^{2m})sw)$ & $(1-z^{\frac m3})z^{\frac13}$, $\zeta_3(1-\zeta_3z^{\frac m3})z^{\frac13}$\\
	\hline
\end{tabular}
\end{center}

Recall that $m\equiv 1$ $(\operatorname{mod}3)$ for type $\mathrm{II}$ and $\mathrm{IV}^*$; $m\equiv 2$ $(\operatorname{mod}3)$ for type $\mathrm{II}^*$ and $\mathrm{IV}$; $m\equiv 1$ $(\operatorname{mod}2)$ for type $\mathrm{III}$ and $\mathrm{III}^*$.

\noindent \textbf{Case 2: $\mathrm{I}_b$ singular fiber} Let $\pi: U \rightarrow \Delta$ be an elliptic fibration with $\mathrm{I}_1$ singular fiber at $0$. According to Kodaira, the map
\begin{equation*}
\begin{aligned}
\Psi:(\Delta^*\times \mathbb{C})/\left(\mathbb{Z} \tau_1+\mathbb{Z} \tau_2\right)\to \Delta_z\times \mathbb{P}^2_{[X:Y:W]},\; (z,v)\mapsto   \left(z,\left[-\frac{1}{12}-\frac{1}{4 \pi^2} \wp_z(v): \frac{i}{8 \pi^3} \wp_z^{\prime}(v):1\right]\right)
\end{aligned}
\end{equation*}is an isomorphism between 
\begin{equation*}
\left(\Delta^* \times \mathbb{C}\right) /\left(\mathbb{Z} \tau_1+\mathbb{Z} \tau_2\right)\cong \{(z,[X:Y:W])|Y^2W=4 X^3+X^2W-g_2(z) XW^2-g_3(z)W^3\}\subset \Delta^* \times \mathbb{P}^2
\end{equation*}
where $\tau_1=1$, $\tau_2=\frac{1}{2\pi i}\log z$, $g_2(z)=20\sum\limits_{n=1}^\infty (1-z^n)^{-1}n^3z^n$, $g_3(z)=\frac13\sum\limits_{n=1}^\infty (1-z^n)^{-1}(7n^5+5n^3)z^n$ are regular on $\Delta$ with $g_2(0)=g_3(0)=0$ and $\wp_z(v)=\frac{1}{v^2}+\sum\limits_{\lambda\in \Lambda(z)\backslash (0,0)}\left(\frac{1}{(v-\lambda)^2}-\frac{1}{\lambda^2}\right)$ is the Weierstrass $\wp$-function associated with the lattice $\Lambda(z)$ generated by $\tau_1(z), \tau_2(z)$. $\operatorname{Im}\Psi$ in $\Delta \times \mathbb{P}^2$ is a smooth elliptic surface $\bar{U}$ with central fiber the node, $y^2=4 x^3+x^2$. This $\bar{U}$ represents the Kodaira canonical form for a singular fiber of type $\mathrm{I}_1$.

Recall that $\frac{dx}{y}$ is the generating holomorphic $1-$form on an elliptic curve $E=\{[X:Y:W]\in \mathbb{P}^2|Y^2W=4 X^3+X^2W-g_2(z) XW^2-g_3(z)W^3\}$, where $x=\frac{X}{Z}$, $y=\frac{Y}{Z}$. From
\begin{align*}
\Psi^*\left(dz\wedge \frac{dx}{y}\right)=dz\wedge \frac{-\frac{1}{4\pi^2}d\wp_z(v)}{\frac{i}{8\pi^3}\wp_z'(v)}=2\pi i dz\wedge dv,
\end{align*}we know that when changing coordinates from $(z,(x,y))$ to $(z,v)$, the pullback form $\Psi^{-1*}(dz\wedge dv)$ has no order along the singular fiber, that is $N=0$.

An $\mathrm{I}_b$ singular fiber can always be realized over $\Delta^*$ with $\tau_1=1, \tau_2=\frac{b}{2 \pi i} \log z$. It admits an unramified, fiber-preserving $b$-fold covering map onto an $\mathrm{I}_1$ degeneration globally. So in this case $N=0$.

\noindent \textbf{Case 3: $\mathrm{I}^*_b$ singular fiber:} Let $\pi: U \rightarrow \Delta$ be an elliptic fibration with $\mathrm{I}^*_b$ singular fiber at $0$. The pullback of $\left.U\right|_{\Delta^*}$ under $z=u^2$ extends as $U^{\prime} \rightarrow \Delta_u$ with an $\mathrm{I}_{2 b}$ central fiber. $\mathbb{Z}_2$ acts naturally on $U^{\prime}$, extending the action by deck transformations on $U^{\prime}$, in such a way that the quotient elliptic surface has four ordinary double points along its central fiber. Resolving these singularities results in the Kodaira canonical form $\bar{U}$ for $I_b^*$. We can assume that $\left.U^{\prime}\right|_{\Delta^*}=\left(\Delta_u^* \times \mathbb{C}_w\right) /\left(\mathbb{Z} \tau_1+\mathbb{Z} \tau_2\right)$ with $\tau_1=1$ and $\tau_2=\frac{b}{\pi i} \log u$. Since $d u \wedge d w$ is invariant under the $\mathbb{Z}_2$-action, it induces a holomorphic volume form on the crepant resolution $\bar{U}$. 

Now we can change the coordinate 
\begin{equation*}
\Delta^*_u\times \mathbb{C}_w\to \Delta^*_z\times \mathbb{C}_v,\; (u,w)\mapsto (z,v)=(u^2,uw).
\end{equation*}In $(z,v)$ coordinate, we have $\tau_1(z)=z^{\frac12}$, $\tau_2(z)=\frac{b}{2\pi i}z^{\frac12}\log z$, $dz\wedge dv=u^2du\wedge dw$. Therefore, when changing coordinate from $(u,w)$ to $(z,v)$, the pullback of $dz\wedge dv$ has order $N=1$ along the singualar fiber.

\subsection{Ansatz on fiber product}

Now we prove part of Theorem \ref{main 2}. In this section, we only discuss the global algebraic structure, construct the ansatz modeled at infinity and verify the asymptotic behaviors (including the verification of $\mathrm{SOB}$ and $\mathrm{HMG}$ properties). We will discuss gluing issues and perturbation to a genuine Calabi-Yau metric in the next section.

The analysis is largely analogous across all these cases, so we will focus on a few representative examples.

\noindent \textbf{Situation 1:} Both $F_1$ and $F_2$ are of finite monodromy, i.e. corresponding to (1) in Theorem \ref{main 2}.

For (1a), (1b), (1c), we consider the case when $F_1$ is a singular fiber of type $\mathrm{II}^*$ and $F_2$ is of type $\mathrm{III}^*$. 

(1a) Let $U_i$ be a neighborhood of singular fiber $F_i$. We pull back $U_1|_{\Delta^*}$ under $z_1=s_1^{6}$, $U_2|_{\Delta^*}$ under $z_2=s_2^{4}$. Let $U_i'$ be the pullback of $U_1|_{\Delta^*}$ under $z_i$. From discussion in Section \ref{elaboration}, we may conclude
\begin{equation*}
U_1'=(\Delta^*_{s_1}\times \mathbb{C}_{w_1})/(\mathbb{Z}+\mathbb{Z}\tau(s_1^6)),
\end{equation*}and 
\begin{equation*}
\mathscr{A}_1:(s_1,w_1)\mapsto \left(\zeta_6s_1,\zeta_6^5\frac{1-s_1^{2m_1}}{1-\zeta_3s_1^{2m_1}}w_1\right);
\end{equation*}
\begin{equation*}
U_2'=(\Delta^*_{s_2}\times \mathbb{C}_{w_2})/(\mathbb{Z}+\mathbb{Z}\tilde{\tau}(s_2^4)),
\end{equation*}and 
\begin{equation*}
\mathscr{A}_2:(s_2,w_2)\mapsto \left(i s_2,-i\frac{1-s_2^{2m_2}}{1+s_2^{2m_2}}w_2\right),
\end{equation*}where $\tau(z_1)=\zeta_3\frac{1-\zeta_3z_1^{\frac{m_1}{3}}}{1-z_1^{\frac{m_1}{3}}}$, $\tilde{\tau}(z_2)=i\frac{1+z_2^{\frac{m_2}{2}}}{1-z_2^{\frac{m_2}{2}}}$. When taking fiber product, we have $z=z_1=z_2$. It is therefore convenient to perform a further pullback with $s_1=s^2$, $s_2=s^3$. Consequently, the fiber product space $U|_{\Delta^*}$ when pulled back under $z=s^{12}$, can be expressed as $U'=(\Delta^*_s\times \mathbb{C}_{w_1}\times \mathbb{C}_{w_2})/\mathbb{Z}+\mathbb{Z}\tau(s_1^6)+\mathbb{Z}+\mathbb{Z}\tilde{\tau}(s_2^4)$, where the deck action of monodromy acts on $U'$ by 
\begin{equation*}
\mathscr{A}:(s,w_1,w_2)\mapsto \left(\zeta_{12}s,\zeta_6^5\frac{1-s^{4m_1}}{1-\zeta_3s^{4m_1}}w_1, -i\frac{1-s^{6m_2}}{1+s^{6m_2}}w_2\right).
\end{equation*}
Hence $U|_{\Delta^*}=U'/\sim_{\mathscr{A}}$.

(1b) First we describe the compactification. The map
\begin{equation*}
\Phi:\Delta^*\times \mathbb{C}\times \mathbb{C}\to \Delta^*\times \mathbb{C}\times \mathbb{C},\; (s,w_1,w_2)\mapsto \left(s^{12},(1-s^{4m_1})s^{2}w_1, (1-s^{6m_2})s^3w_2\right)
\end{equation*}factors through the monodromy action $\mathscr{A}$ and hence induces an isomorphism of elliptic fibrations
\begin{align*}
\Phi:U|_{\Delta^*}=U'/\sim_{\mathscr{A}}\longrightarrow (\Delta^*\times \mathbb{C}\times \mathbb{C})/(\mathbb{Z}\tau_1+\mathbb{Z}\tau_2+\mathbb{Z}\tau_3+\mathbb{Z}\tau_4),
\end{align*}where 
\begin{align*}
&\tau_1(z)=\operatorname{pr}_2\Phi\left(z^{\frac13},1\right)=\left(1-z^{\frac{m_1}{3}}\right)z^{\frac16};\\
&\tau_2(z)=\operatorname{pr}_2\Phi\left(z^{\frac13},\zeta_3\right)=\zeta_3\left(1-\zeta_3z^{\frac{m_1}{3}}\right)z^{\frac16};\\
&\tau_3(z)=\operatorname{pr}_3\Phi\left(z^{\frac13},1\right)=\left(1-z^{\frac{m_2}{2}}\right)z^{\frac14};\\
&\tau_4(z)=\operatorname{pr}_3\Phi\left(z^{\frac13},\zeta_3\right)=i\left(1+z^{\frac{m_2}{2}}\right)z^{\frac14}.
\end{align*}
Moreover, $(z,v_1,v_2)=\left(s^{12},(1-s^{4m_1})s^{2}w_1, (1-s^{6m_2})s^3w_2\right)$ can be viewed as a coordinate on $U|_{\Delta^*}$. Note that $U'$ can be extended to $s=0$ and the central fiber is $\mathbb{C}^2/(\mathbb{Z}+\zeta_3\mathbb{Z}+\mathbb{Z}+i\mathbb{Z})$. Denote $\widetilde{U'}$ to be the extended space. The deck action $\mathscr{A}$ also acts on $\widetilde{U'}$, though the action is not free anymore. Then $\widetilde{U'}/\sim_{\mathscr{A}}$ provides a compactification of $U|_{\Delta^*}$. Hence we can compactify $M$ into a new projective variety $\widetilde{X}$.

Then we try to calculate the canonical divisor of $\widetilde{X}$. We first find a holomorphic volume form $\Omega$ on $M$, and then extend it to a rational form, still denoted $\Omega$ on $\widetilde{X}$. Since $\widetilde{X}$ is relatively minimal, we can define an order of pole along the singular fiber. From this we obtain the canonical divisor.

From characterization of rational elliptic surface, we know that there exists holomorphic volume form $\Omega_i$ on $X_i\backslash F_i$, and $\Omega_i$ has an order one pole along $F_i$. According to Hein \cite{hein2010gravitational}, in the local coordinate $(z_i,v_i)$ the holomorphic volume form is of the form $\Omega_i=\frac{k_i(z_i)}{z_i^2}dz_i\wedge dv_i$, where $i=1,2$. We use Poincar\'{e}-Residue formula to construct a holomorphic volume form on $X\backslash F$. Now $\Omega_1\wedge \Omega_2=-\frac{k_1(z_1)k_2(z_2)}{z_1^2z_2^2}dz_1\wedge dz_2\wedge dv_1\wedge dv_2$ can be viewed as a global section of $K_{X_1\times X_2}$.  The fiber product $X_1\times_{\mathbb{P}^1}X_2=\{(x_1,x_2)\in X_1\times X_2|\pi_1(x_1)=\pi_2(x_2)\}$ can be viewed as a hypersurface or divisor in the product space $X_1\times X_2$. Also note that there is a commutative diagram
\begin{displaymath}
\xymatrix{X_1\times_{\mathbb{P}^1}X_2\ar[r]^{\iota}\ar[d]&X_1\times X_2\ar[d]\\
D\ar[r]^{\iota}& \mathbb{P}^1\times \mathbb{P}^1},
\end{displaymath}where $D\subset \mathbb{P}^1\times \mathbb{P}^1$ is the diagonal. Since $D\sim\{p\}\times \mathbb{P}^1+\mathbb{P}^1\times \{p\}$ for any point $p$ in $\mathbb{P}^1$, we have $\mathcal{O}_{X_1\times X_2}(X)\sim \pi_1^*\mathcal{O}_{\mathbb{P}^1}(1)\boxtimes\pi_2^*\mathcal{O}_{\mathbb{P}^1}(1)$. Let $p$ be the point corresponding to $F_i$. In the local coordinate $(z_1,z_2,v_1,v_2)$, the global section $s$ of $\mathcal{O}_{X_1\times X_2}(X)$, when restricted to $U_1\times U_2$, is of the form $\frac{h(z_1,z_2,v_2,v_2)z_1z_2}{z_1-z_2}$, where $h$ is a local holomorphic nonzero function on $U_1\times U_2$, $z_1z_2$ corresponds to $\pi_1^*\mathcal{O}_{\mathbb{P}^1}(1)\boxtimes\pi_2^*\mathcal{O}_{\mathbb{P}^1}(1)$ and $\frac{1}{z_1-z_2}$ is a local section of $\mathcal{O}_{X_1\times X_2}(X)$. From Poincar\'{e} residue formula, we get
\begin{equation*}
\Omega_1\wedge \Omega_2\otimes s=-\frac{k_1k_2}{z_1^2z_2^2}\frac{hz_1z_2}{z_1-z_2}d(z_1-z_2)\wedge dz_2\wedge dv_1\wedge dv_2\mapsto \Omega \coloneqq -\frac{k_1k_2h}{z^2}dz\wedge dw_1\wedge dw_2.
\end{equation*}The well-definedness of Poincar\'{e} residue suggests that $\Omega$ is a global holomorphic volume form on $M=X\backslash F$ (not on $X$!). Since $X\backslash F\to \mathbb{P}^1$ is an abelian surface fibration, if we apply Poincar\'{e} residue formula to a generic smooth fiber, we map $\Omega$ to the holomorphic volume form on an abelian surface. This implies $h(z,v_1,v_2)=h(z)$ depends only on $z$. In conclusion, $\Omega=\frac{k(z)}{z^2}dz\wedge dv_1\wedge dv_2$ on $M$.

Now we try to extend $\Omega$ to $\widetilde{X}$. To do so, we need to calculate the transition function between coordinate $(z,v_1,v_2)$ on $M$ and the coordinate on $\widetilde{X}$.

In the neighborhood of the singular point $(0,0,0)$, if we take $w_1'=\left(1-s_1^{4m_1}\right)w_1$, $w_2'=\left(1-s_2^{6m_2}\right)w_2$, then the monodromy action in this new coordinate looks like
\begin{equation*}
\mathscr{A}:(s,w_1',w_2')\mapsto (\zeta_{12}s,\zeta_6^5w_1',-iw_2').
\end{equation*}
Note that $$\mathbb{C}^3/\mathbb{Z}_{12}=\operatorname{Spec}\mathbb{C}[s^{12},s^{2}w_1',s^3w_2',w_1^{'6},w_2^{'4},sw_1^{'5}w_2',\cdots].$$ We can take $x_1=w_1^{'6}$, $x_2=w_2^{'4}$, $x_3=sw_1^{'5}w_2'$ as local coordinte on $\widetilde{X}$ across the singular fiber. Now $(x_1,x_2,x_3)\mapsto x_3^3$ represents the fibration map and $\{x_3=0\}$ represents the reduced component of the fiber. 

Direct computation shows that
\begin{align*}
\frac{1}{z^2}dz\wedge dv_1\wedge dv_2=12\frac{1}{s^8}\left(1-S_1^{4m_1}\right)\left(1-s_2^{6m_2}\right)ds\wedge dw_1\wedge dw_2=\frac{h(x_1,x_2)}{x_3^8}dx_1\wedge dx_2\wedge dx_3,
\end{align*}where $h(x_1,x_2)$ is some holomorphic function. From this we know that the canonical divisor of $\widetilde{X}$ is 
\begin{equation*}
K_{\widetilde{X}}=-\frac{8}{12}[\widetilde{F}]=\frac{12-10-9-1}{12}[\widetilde{F}].
\end{equation*}

(1c) Now we calculate the semi-flat ansatz by using Corollary \ref{semi-flat 6}. From the expressions of $\tau_i(z)$, $i=1,2,3,4$, we immediately get
\begin{align*}
\operatorname{Im}(\bar{\tau}_1\tau_2)=\frac{\sqrt{3}}{2}\left(1-|z|^{\frac{2m_1}{3}}\right)|z|^{\frac13},\quad \operatorname{Im}(\bar{\tau}_3\tau_4)=\left(1-|z|^{m_2}\right)|z|^{\frac12};
\end{align*}
\begin{align*}
\Gamma^1(z,v)&=\frac{1}{\operatorname{Im}\left(\bar{\tau}_1 \tau_2\right)}\left(\operatorname{Im}\left(\bar{\tau}_1 v_1\right) \tau_2^{\prime}-\operatorname{Im}\left(\bar{\tau}_2 v_1\right) \tau_1^{\prime}\right)\\
&=\frac{1}{6}\left(\frac{v_1}{z}+\mathrm{Error}_1\right);
\end{align*}
\begin{align*}
\Gamma^2(z, v)&=\frac{1}{\operatorname{Im}\left(\bar{\tau}_3 \tau_4\right)}\left(\operatorname{Im}\left(\bar{\tau}_3 v_2\right) \tau_4^{\prime}-\operatorname{Im}\left(\bar{\tau}_4 v_2\right) \tau_3^{\prime}\right)\\
&=\frac14\left(\frac{v_2}{z}+\mathrm{Error}_2\right),
\end{align*}where
\begin{align*}
&\mathrm{Error}_1=\frac{-2m_1}{v_1}\frac{z^{\frac{m_1}{3}}}{1-|z|^{\frac{2m_1}{3}}}\left\{\bar{z}^{\frac{m_1}{3}}v_1+\frac{z^{\frac16}}{\bar{z}^{\frac16}}\bar{v}_1\right\},\\
&\mathrm{Error}_2=\frac{2m_2}{v_2}\frac{z^{\frac{m_2}{2}}}{1-|z|^{m_2}}\left\{v_2+\frac{z^{\frac12}}{|z|^{\frac12}}\bar{v}_2\right\}
\end{align*}
Plugging all these data into the formula \eqref{semi-flat 7}, we get
\begin{align*}
\omega_{\mathrm{sf}, \varepsilon}&=\frac{2\sqrt{3}i|k|^2\left(1-|z|^{\frac{2m_1}{3}}\right)\left(1-|z|^{m_2}\right)|z|^{\frac56}}{\varepsilon^2}\frac{dz\wedge d\bar{z}}{|z|^4}\\
&+\frac{i}{2}\frac{\varepsilon}{\frac{\sqrt{3}}{2}\left(1-|z|^{\frac{2m_1}{3}}\right)|z|^{\frac13}}\left(d v_1-\frac16 \left(\frac{v_1}{z}+\mathrm{Error}_1\right) d z\right) \wedge\left(d \bar{v}_1-\frac16 \left(\frac{\bar{v}_1}{\bar{z}}+\overline{\mathrm{Error}_1}\right) d \bar{z}\right)\\
&+\frac{i}{2}\frac{\varepsilon}{\left(1-|z|^{m_2}\right)|z|^{\frac12}}\left(d v_2-\frac14 \left(\frac{v_2}{z}+\mathrm{Error}_2\right) d z\right) \wedge\left(d \bar{v}_2-\frac14 \left(\frac{\bar{v}_2}{\bar{z}}+\overline{\mathrm{Error}_2}\right) d \bar{z}\right),
\end{align*}

To analyze the asymptotic behavior of the ansatz, change coordinate on $\Delta\backslash [0,1)$ by setting 
\begin{align*}
z=\left(\frac{\alpha}{\alpha_0}\right)^{-\frac{12}{7}},\; v_1=\left(\frac{\alpha}{\alpha_0}\right)^{-\frac{2}{7}},\; v_2=\left(\frac{\alpha}{\alpha_0}\right)^{-\frac{3}{7}},
\end{align*}where $\alpha_0=\frac{24\sqrt[4]{3}|k(0)|}{7\varepsilon}$.
Then $\alpha$ ranges over the sector $|\alpha|>\alpha_0$, $0<\arg \alpha<\frac{7\pi }{6}$, and $\beta_1$ moves in a fundamental region converging to the torus $\mathbb{C}/\mathbb{Z}+\zeta_3\mathbb{Z}$, $\beta_2$ moves in a fundamental region converging to the torus $\mathbb{C}/\mathbb{Z}+i\mathbb{Z}$ as $|\alpha|\to \infty$. In the new coordinate, the above ansatz looks like
\begin{equation}
\omega_{\operatorname{sf},\varepsilon}=\frac{i}{2}\left\{d\alpha\wedge d\bar{\alpha}+\frac{2\varepsilon}{\sqrt{3}}d\beta_1\wedge d\bar{\beta}_1+\varepsilon d\beta_2\wedge d\bar{\beta}_2\right\}\left(1+\mathrm{Error}(\alpha)\right).
\end{equation}$\operatorname{Error}(\alpha)$ constitutes three error terms $\frac{\left|k\left(\left(\frac{\alpha}{\alpha_0}\right)^{-\frac{12}{7}}\right)\right|^2}{|k(0)|^2}\left(1-\left|\frac{\alpha}{\alpha_0}\right|^{-\frac{8}{7}m_1}\right)\left(1-\left|\frac{\alpha}{\alpha_0}\right|^{-\frac{12}{7}m_2}\right)-1$, $\mathrm{Error}_1(\alpha)$ and $\mathrm{Error}_2(\alpha)$. Hence, $\operatorname{Error}(\alpha)=O\left(|\alpha|^{-\frac{12}{7}}\right)$ and the error bound improves by a factor $|\alpha|^{-1}$ whenever we take a derivative with respect to $\alpha$ or $\beta$. In particular, $|\operatorname{Rm}|=O\left(|\alpha|^{-\frac{31}{12}}\right)$.

(1d) There are only three cases, that is $(F_1,F_2)=(\mathrm{II},\mathrm{II}^*)$, $(\mathrm{III},\mathrm{III}^*)$, $(\mathrm{IV},\mathrm{IV}^*)$. We treat the case $(\mathrm{III},\mathrm{III}^*)$ in detail.

The fiber product space $U|_{\Delta^*}$, when pulled back under $z=s^{4}$, can be written as $U'=(\Delta^*_s\times \mathbb{C}_{w_1}\times \mathbb{C}_{w_2})/\mathbb{Z}+\mathbb{Z}\tau(s^4)+\mathbb{Z}+\mathbb{Z}\tilde{\tau}(s^4)$, where the deck action of monodromy acts on $U'$ by 
\begin{equation*}
\mathscr{A}:(s,w_1,w_2)\mapsto \left(i s,i\frac{1-s^{2m_1}}{1+s^{2m_1}}w_1, -i\frac{1-s^{2m_2}}{1+2^{6m_2}}w_2\right).
\end{equation*}
Hence $U|_{\Delta^*}=U'/\sim_{\mathscr{A}}$. The map
\begin{equation*}
\Phi:\Delta^*\times \mathbb{C}\times \mathbb{C}\to \Delta^*\times \mathbb{C}\times \mathbb{C},\; (s,w_1,w_2)\mapsto \left(s^{4},(1-s^{2m_1})s^{3}w_1, (1-s^{2m_2})sw_2\right)
\end{equation*}factors through the monodromy action $\mathscr{A}$ and hence induces an isomorphism of fibrations
\begin{align*}
\Phi:U|_{\Delta^*}=U'/\sim_{\mathscr{A}}\longrightarrow (\Delta^*\times \mathbb{C}\times \mathbb{C})/(\mathbb{Z}\tau_1+\mathbb{Z}\tau_2+\mathbb{Z}\tau_3+\mathbb{Z}\tau_4),
\end{align*}where 
\begin{align*}
&\tau_1(z)=\operatorname{pr}_2\Phi\left(z^{\frac14},1\right)=\left(1-z^{\frac{m_1}{2}}\right)z^{\frac34};\\
&\tau_2(z)=\operatorname{pr}_2\Phi\left(z^{\frac14},3\right)=i\left(1+z^{\frac{m_1}{2}}\right)z^{\frac34};\\
&\tau_3(z)=\operatorname{pr}_3\Phi\left(z^{\frac14},1\right)=\left(1-z^{\frac{m_2}{2}}\right)z^{\frac14};\\
&\tau_4(z)=\operatorname{pr}_3\Phi\left(z^{\frac14},i\right)=i\left(1+z^{\frac{m_2}{2}}\right)z^{\frac14}.
\end{align*}
Moreover, $(z,v_1,v_2)=\left(s^{12},(1-s^{4m_1})s^{2}w_1, (1-s^{6m_2})s^3w_2\right)$ can be viewed as a coordinate on $U|_{\Delta^*}$. Note that $U'$ can be extended to $s=0$ and the central fiber is $\mathbb{C}^2/(\mathbb{Z}+\zeta_3\mathbb{Z}+\mathbb{Z}+i\mathbb{Z})$. Denote $\widetilde{U'}$ to be the extended space. The deck action $\mathscr{A}$ also acts on $\widetilde{U'}$, though the action is not free anymore. Then $\widetilde{U'}/\sim_{\mathscr{A}}$ provides a compactification of $U|_{\Delta^*}$. Hence we can compactify $M$ into a new projective variety $\widetilde{X}$. Similarly, we know that $K_{\widetilde{X}}=-\frac14 [\widetilde{F}]$. The only difference is the asymptotic behavior of the corresponding ansatz. Plugging all the data into the formula \eqref{semi-flat 7}, we get
\begin{align*}
\omega_{\mathrm{sf}, \varepsilon}&=\frac{4i|k|^2\left(1-|z|^{m_1}\right)\left(1-|z|^{m_2}\right)}{\varepsilon^2}\frac{dz\wedge d\bar{z}}{|z|^2}\\
&+\frac{i}{2}\frac{\varepsilon}{\left(1-|z|^{m_1}\right)|z|^{\frac32}}\left(d v_1-\frac34 \left(\frac{v_1}{z}+\mathrm{Error}_1\right) d z\right) \wedge\left(d \bar{v}_1-\frac34 \left(\frac{\bar{v}_1}{\bar{z}}+\overline{\mathrm{Error}_1}\right) d \bar{z}\right)\\
&+\frac{i}{2}\frac{\varepsilon}{\left(1-|z|^{m_2}\right)|z|^{\frac12}}\left(d v_2-\frac14 \left(\frac{v_2}{z}+\mathrm{Error}_2\right) d z\right) \wedge\left(d \bar{v}_2-\frac14 \left(\frac{\bar{v}_2}{\bar{z}}+\overline{\mathrm{Error}_2}\right) d \bar{z}\right),
\end{align*}where
\begin{align*}
&\mathrm{Error}_1=\frac{2m_1}{v_1}\frac{z^{\frac{m_1}{2}}}{1-|z|^{m_1}}\left\{v_1+\frac{z^{\frac32}}{|z|^{\frac32}}\bar{v}_1\right\},\\
&\mathrm{Error}_2=\frac{2m_2}{v_2}\frac{z^{\frac{m_2}{2}}}{1-|z|^{m_2}}\left\{v_2+\frac{z^{\frac12}}{|z|^{\frac12}}\bar{v}_2\right\}
\end{align*}
Change the coordinate on $\Delta\backslash [0,1)$ by setting $z=\exp\left(-\frac{\varepsilon \alpha}{2\sqrt{2}|k(0)|}\right)$, $v_1=\exp\left(-\frac{3\varepsilon \alpha}{8\sqrt{2}|k(0)|}\right)\beta_1$, $v_2=\exp\left(-\frac{\varepsilon \alpha}{8\sqrt{2}|k(0)|}\right)\beta_2$, where $\alpha\in (0,\infty)\times i\left(0,\frac{4\sqrt{2}\pi|k(0)|}{\varepsilon}\right)$, $\beta_i$ move in a fundamental region converging to the torus $\mathbb{C}/\mathbb{Z}+i\mathbb{Z}$ as $|\alpha|\to \infty$. Then the metric in the new coordinate is of the form
\begin{align*}
\omega_{\operatorname{sf},\varepsilon}=\frac{i}{2}\left(d\alpha\wedge d\bar{\alpha}+d\beta_1\wedge d\bar{\beta}_1+d\beta_2\wedge d\bar{\beta}_2\right)\left(1+\mathrm{Error}(\alpha)\right),
\end{align*}where $\mathrm{Error}(\alpha)=O_{\varepsilon}\left(\exp\left(-\frac{\varepsilon \operatorname{Re}\alpha}{2\sqrt{2}|k(0)|}\right)\right)$. This error bound holds true for all derivatives as well.

\noindent \textbf{Situation 2:} Both $F_i$ are of $\mathrm{I}^*_{b_i}$ type, i.e. corresponding to (2) in Theorem \ref{main 2}. 

(2a) Let $U_i$ be a neighborhood of singular fiber $F_i$. Let $\Omega_i$ be the holomorphic volume form on $X_i\backslash F_i$. From the discussion in Section \ref{elaboration}, we know that there exists an appropriate coordinate $(z_i,v_i)$ on $U_i$ such that the $\Omega_i=\frac{k_i(z)}{z_i^2}dz_i\wedge dv_i$, $i=1,2$, where $k_i(0)\neq 0$. According to the Poincar\'{e} residue formula, we can construct a holomorphic volume form $\Omega$ on $X\backslash F$ such that in the fiber product coordinate, we have $\Omega=\frac{k(z)}{z^2}dz\wedge dv_1\wedge dv_2$.

We then describe the compactification of $X$. Pull back $U_i|_{\Delta^*}$ under $z_i=u_i^2$, the resulting $U_i'$ can be extended to $\widetilde{U_i'}$ such that the singular fiber at $0$ is of $\mathrm{I}_{2b_i}$-type. Consider the fiber product $\widetilde{U}=\widetilde{U_1'}\times_{\mathbb{P}^1}\widetilde{U_2'}$. It's still a normal projective variety, since there are only $4b_1b_2$ isolated singular points on the fiber product. Now we can quotient $\widetilde{U}$ by the natural $\mathbb{Z}_2$-action. It's clear that outside the singular fiber, $\widetilde{U}/\mathbb{Z}_2$ is isomorphic to the fiber product of $U_1\backslash F_1\times_{\mathbb{P}^1}U_2\backslash F_2$. Hence $\widetilde{U}/\mathbb{Z}_2$ is a natural compactification. Note that $(z,v_1,v_2)=(u^2,uw_1,uw_2)$, where $(z,w_1)$, $(z,w_2)$ are coordinates on $U_1'$ and $U_2'$ respectively and $(z,v_1,v_2)$ is the coordinate on $\widetilde{U}$. Now
\begin{equation*}
\Omega=\frac{k(z)}{z^2}dz\wedge dv_1\wedge dv_2=\frac{2k(u^2)}{u}du\wedge dw_1\wedge dw_2. 
\end{equation*}Hence $K_{\widetilde{X}}=-\frac12 [\widetilde{F}]$.

(2b) The period in the coordinate of $(z,v_1,v_2)$ is 
\begin{align*}
\tau_1(z)=z^{\frac12},\; \tau_2(z)=\frac{b_1}{2\pi i}z^{\frac12}\log z,\; \tau_3(z)=z^{\frac12},\; \tau_4(z)=\frac{b_2}{2\pi i}z^{\frac12}\log z.
\end{align*}Plugging these data into formula \eqref{semi-flat 7}, we get
\begin{align*}
\omega_{\operatorname{sf},\varepsilon}&=\frac{ib_1b_2|k|^2\left|\log |z|\right|^2}{\pi^2\varepsilon^2|z|^2}dz\wedge d\bar{z}\\
&+\frac{i\pi \varepsilon}{4|z|\left|\log |z|\right|}\sum\limits_{j=1}^2\frac{1}{b_j}\left\{dv_j-\left(\frac{v_j}{z}-\frac{i\operatorname{Im}\left(z^{-\frac12}v_j\right)}{4z^{\frac12}\left|\log |z|\right|} \right)dz\right\}\wedge \left\{d\bar{v}_j-\left(\frac{\bar{v}_j}{\bar{z}}+\frac{i\operatorname{Im}\left(z^{-\frac12}v_j\right)}{4\bar{z}^{\frac12}\left|\log |z|\right|} \right)d\bar{z}\right\}
\end{align*}
For the purpose of analyzing the asymptotic behavior, it is beneficial to reformulate the ansatz using the coordinates  $(u,w_1,w_2)$:
\begin{align*}
\omega_{\operatorname{sf},\varepsilon}&=\frac{4i|k|^2}{\pi^2\varepsilon^2}\frac{\left|\log |u|\right|^2}{|u|^4}du\wedge d\bar{u}\\
&+\frac{i\pi\varepsilon}{2\left|\log |u|\right|}\sum\limits_{j=1}^2\frac{1}{b_j}\left(dw_j+\frac{i\operatorname{Im}(w_j)}{u\left|\log |u|\right|}du\right)\wedge \left(d\bar{w}_j-\frac{i\operatorname{Im}(w_j)}{\bar{u}\left|\log |u|\right|}d\bar{u}\right).
\end{align*}

Let $\Delta^*=\left\{z\in X|0<z<\frac12\right\}$, $U=f^{-1}(\Delta^*)$, $z_x=f(x)$ for $x\in U$. For fixed $x_0\in U$, define $r_x=\operatorname{dist}_M(x_0,x)$. Let $g=\frac{i|k|^2\left|\log |z|\right|^2}{\pi^2\varepsilon^2|z|^2}dz\wedge d\bar{z}$.

Note that the diameter of $f^{-1}(z)$ is approximately $\frac{C\left|\log|u|\right|}{\sqrt{\left|\log|u|\right|}}=C\left|\log|u|\right|^{\frac12}=C'\left|\log|z|\right|^{\frac12}$; The $g-$length of $\{z'\in \Delta^*:|z'|=|z|\}$ is approximately $2\pi \frac{\left|\log|z|\right|}{|z|}|z|=2\pi \left|\log|z|\right|$; For fixed $z_0$ and any $z$ lying on the line $0z_0$, the radial distance from $z_0$ to $z$ is approximately $\left|\int_{|z_0|}^{|z|}\frac{\log r}{r}dr\right|\sim \frac12 \left|\log |z|\right|^2$ as $z\to 0$. Thus for any fixed $x_0\in U$ and any $x\in U$ sufficiently close to the singular fiber, we have 
\begin{align*}
r_x\sim \operatorname{dist}_g(z_x,z_{x_0})\sim \left|\log |z_x|\right|^2.
\end{align*}

Note $f:U\to \Delta^*$ is a Riemannian submersion whose fiber volume is $\varepsilon$, so $\operatorname{Vol}(f^{-1}(B),g_{\operatorname{sf},\varepsilon})=\varepsilon\cdot \operatorname{Vol}(B,g)$, for any $B\subset \Delta^*$. Moreover, $f$ is bi-Lipschitz, hence 
\begin{align*}
\operatorname{Vol}_{g_{\operatorname{sf},\varepsilon}}(B(x_0,r_x))\sim C\int_0^{2\pi}d\theta\int_{|z_{x_0}|}^{|z_x|}\frac{(\log r)^2}{r^2}rdr\sim C \left|\log|z_x|\right|^3\sim r_x^{\frac 32},
\end{align*}i.e. the volume growth of geodesic ball is of order $\frac32$. Hence $\omega_{\operatorname{sf},\varepsilon}$ is $\operatorname{SOB}\left(\frac32\right)$.

Next, we prove $\left|B\left(x,\frac12 r_x\right)\right|\geq \frac{1}{C}r_x^{\frac32}$, if $r_x>>1$. Recall that the diameter of $f^{-1}(z)$ and the $g$-length of $\{z'\in \Delta:|z'|=|z|\}$ are of order  $\left|\log|z|\right|^{\frac12}$, $\left|\log|z|\right|$ respectively, so 
\begin{equation*}
B\left(x,\frac12 r_x\right)\supset\left\{y\in U:|z_x|<|z_y|<|z_x|^{1-\alpha}\right\},
\end{equation*}for all $x\in U$, $r_x>>1$ and small $\alpha>0$. By coarea formula, we have
\begin{align*}
\left|B\left(x,\frac12 r_x\right)\right|\geq 2\pi\int_{|z_x|}^{|z_x|^{1-\alpha}}r^{-2}|\log r|dr=\frac{2\pi }{3}(\alpha^3-3\alpha^2+3\alpha)\left|\log |z_x|\right|^3\sim r_x^{\frac32}.
\end{align*}

Hence $\omega_{\operatorname{sf},\varepsilon}$ is $\operatorname{SOB}\left(\frac32\right)$.

To determine the norm of Riemann curvature tensor, we just need to calculate the norm of Chern curvature. Direct computation shows that
\begin{equation*}
|\operatorname{Rm}(x)|^2\sim \frac{(\pi\varepsilon)^4|z_x|^2}{4|k(0)|^4|\log |z_x||^4}\sim  e^{-2r_x^{\frac12}} r_x^{-2}.
\end{equation*}From this, we know that $C^{k,\alpha}$ quasi-atlas exists.

Now we compute the asymptotic cone of $M$. Note that the diameter of fiber $f^{-1}(z_x)$ is only of order $r_x^{\frac14}$, hence we only need to compute the tangent cone of $g$. Define
\begin{align*}
\Phi_\lambda:[0,\infty)\times S^1\to \Delta^*,\; (s,\theta)\mapsto \frac12 \operatorname{exp}\left(-\lambda^{-\frac12}s^{\frac12}+i\theta\right)=z. 
\end{align*}Direct computation shows that
\begin{align*}
\Psi^*_\lambda(\lambda^2 g)=\frac{4i|k|^2b\lambda^2\left|\log \frac12-\lambda^{-\frac12}s^{\frac12}\right|^2}{2\pi \varepsilon}\left(\frac14\lambda^{-1}s^{-1}ds\otimes ds+d\theta\otimes d\theta\right)\stackrel{\lambda\to 0}{\longrightarrow} \frac{i\left|k\left(\frac12\right)\right|^2b}{2\pi \varepsilon}ds\otimes ds.
\end{align*}Hence the tangent cone is $\mathbb{R}^+$.

\noindent \textbf{Situation 3:} $F_1$ is of $\mathrm{I}^*_b$ type and $F_2$ is of $\mathrm{II}^*$ or $\mathrm{III}^*$ or $\mathrm{IV}^*$ type, i.e. corresponding to (3) in Theorem \ref{main 2}.

We only discuss the case when $F_2$ is of type $\mathrm{IV}^*$ in detail.

(3a) Let $U_i$ be a neighborhood of singular fiber $F_i$. Let $\Omega_i$ be the holomorphic volume form on $X_i\backslash F_i$. From the discussion in Section \ref{elaboration}, we know that there exists an appropriate coordinate $(z_i,v_i)$ on $U_i$ such that the $\Omega_i=\frac{k_i(z)}{z_i^2}dz_i\wedge dv_i$, $i=1,2$, where $k_i(0)\neq 0$. According to the Poincar\'{e} residue formula, we can construct a holomorphic volume form $\Omega$ on $X\backslash F$ such that in the fiber product coordinate, we have $\Omega=\frac{k(z)}{z^2}dz\wedge dv_1\wedge dv_2$.

We then describe the compactification of $X$. Pull back $U_i|_{\Delta^*}$ under $z_i=s_i^6$, and then the resulting $U_i'$ can be extended to $\widetilde{U_i'}$ such that the singular fiber at $0$ is of $\mathrm{I}_{2b}$-type and $\mathrm{II}^*$-type respectively. Consider the fiber product $\widetilde{U}=\widetilde{U_1'}\times_{\mathbb{P}^1}\widetilde{U_2'}$. It's still a normal projective variety. Quotient $\widetilde{U}$ by the natural $\mathbb{Z}_6$ action. It's clear that outside the singular fiber, $\widetilde{U}/\mathbb{Z}_6$ is isomorphic to the fiber product of $U_1\backslash F_1\times_{\mathbb{P}^1}U_2\backslash F_2$. Hence $\widetilde{U}/\mathbb{Z}_6$ is a natural compactification. Note that $(z,v_1,v_2)=(s^6,s^3w_1,sw_2)$, where $(z,v_1)=(s^6,s^3w_1)$, $(z,v_2)=(s^6,sw_2)$ are coordinates on $U_1'$ and $U_2'$ respectively and $(z,v_1,v_2)$ is the coordinate on $\widetilde{U}$. Now
\begin{equation*}
\Omega=\frac{k(z)}{z^2}dz\wedge dv_1\wedge dv_2=\frac{6k(s^6)}{s^3}ds\wedge dw_1\wedge dw_2. 
\end{equation*}Hence $K_{\widetilde{X}}=-\frac12 [\widetilde{F}]$.

(3b) The period in the coordinate of $(z,v_1,v_2)$ is 
\begin{align*}
\tau_1(z)=z^{\frac12},\; \tau_2(z)=\frac{b}{2\pi i}z^{\frac12}\log z,\; \tau_3(z)=\left(1-z^{\frac{m}{3}}\right)z^{\frac13},\; \tau_4(z)=\zeta_3\left(1-\zeta_3z^{\frac{m}{3}}\right)z^{\frac13}.
\end{align*}
Plugging these data into formula \eqref{semi-flat 7}, we get
\begin{align*}
\omega_{\operatorname{sf},\varepsilon}&=\frac{i\sqrt{3}b|k|^2\left|\log |z|\right|\left(1-|z|^{\frac{2m}{3}}\right)}{\pi\varepsilon^2|z|^{\frac73}}dz\wedge d\bar{z}\\
&+\frac{i\pi \varepsilon}{4b|z|\left|\log |z|\right|}\left\{dv_1-\left(\frac{v_1}{z}-\frac{i\operatorname{Im}\left(z^{-\frac12}v_j\right)}{4z^{\frac12}\left|\log |z|\right|} \right)dz\right\}\wedge \left\{d\bar{v}_1-\left(\frac{\bar{v}_1}{\bar{z}}+\frac{i\operatorname{Im}\left(z^{-\frac12}v_1\right)}{4\bar{z}^{\frac12}\left|\log |z|\right|} \right)d\bar{z}\right\}\\
&+\frac{i}{2}\frac{\varepsilon}{\frac{\sqrt{3}}{2}\left(1-|z|^{\frac{2m}{3}}\right)|z|^{\frac23}}\left(d v_1-\frac13 \left(\frac{v_1}{z}+\mathrm{Error}_1\right) d z\right) \wedge\left(d \bar{v}_1-\frac13 \left(\frac{\bar{v}_1}{\bar{z}}+\overline{\mathrm{Error}_1}\right) d \bar{z}\right)
\end{align*}

Let $\Delta^*=\left\{z\in X|0<z<\frac12\right\}$, $U=f^{-1}(\Delta^*)$, $z_x=f(x)$ for $x\in U$. For fixed $x_0\in U$, define $r_x=\operatorname{dist}_M(x_0,x)$. Let $g=\frac{i|k|^2\left|\log |z|\right|^2}{\pi^2\varepsilon^2|z|^2}dz\wedge d\bar{z}$.

Note that the diameter of $f^{-1}(z)$ is approximately $\frac{C\left|\log|u|\right|}{\sqrt{\left|\log|u|\right|}}=C\left|\log|u|\right|^{\frac12}=C'\left|\log|z|\right|^{\frac12}$; The $g-$length of $\{z'\in \Delta^*:|z'|=|z|\}$ is appraximately $2\pi \frac{\left|\log|z|\right|^{\frac12}}{|z|^{\frac76}}|z|=2\pi \frac{\left|\log|z|\right|^{\frac12}}{|z|^{\frac16}}$; For fixed $z_0$ and any $z$ lying on the line $0z_0$, the radial distance from $z_0$ to $z$ is approximately $\left|\int_{|z_0|}^{|z|}\frac{|\log r|^{\frac12}}{r}\right|\sim \frac{\left|\log |z|\right|^{\frac12}}{|z|^{\frac16}}$ as $z\to 0$. Thus for any fixed $x_0\in U$ and any $x\in U$ sufficiently close to the singular fiber, we have 
\begin{align*}
r_x\sim \operatorname{dist}_g(z_x,z_{x_0})\sim \frac{\left|\log |z_x|\right|^{\frac12}}{|z_x|^{\frac16}}.
\end{align*}

Note $f:U\to \Delta^*$ is a Riemannian submersion whose fiber volume is $\varepsilon$ so $\operatorname{Vol}(f^{-1}(B),g_{\operatorname{sf},\varepsilon})=\varepsilon\cdot \operatorname{Vol}(B,g)$, for any $B\subset \Delta^*$. Moreover $f$ is bi-Lipschitz, hence 
\begin{align*}
\operatorname{Vol}_{g_{\operatorname{sf},\varepsilon}}(B(x_0,s))\sim C\int_0^{2\pi}d\theta\int_{|z_{x_0}|}^{|z_x|}\frac{|\log r|}{r^{\frac73}}rdr\wedge d\theta\sim C \frac{\left|\log|z_x|\right|}{|z_x|^{\frac13}}\sim C r_x^{2},
\end{align*}i.e. the volume growth of geodesic ball is of order $2$.

Next, we prove $\left|B\left(x,\frac12 r_x\right)\right|\geq \frac{1}{C}r_x^{2}$, if $r_x>>1$. Recall that the diameter of $f^{-1}(z)$ and the $g$-length of $\{z'\in \Delta:|z'|=|z|\}$ are of order  $\left|\log|z|\right|^{\frac12}$, $\frac{\left|\log|z|\right|^{\frac12}}{|z|^{\frac16}}$ respectively, so 
\begin{equation*}
B\left(x,\frac12 r_x\right)\supset\left\{y\in U:|z_x|<|z_y|<(1+\alpha)|z_x|,\; \left|\operatorname{arg}\frac{z_y}{z_x}\right|<\alpha\right\},
\end{equation*}for all $x\in U$, $r_x>>1$ and small $\alpha>0$. According to coarea formula,
\begin{align*}
\left|B\left(x,\frac12 r_x\right)\right|\geq 2\pi\int_{|z_x|}^{(1+\alpha)|z_x|}\frac{|\log r|}{r^{\frac73}}rdr\sim \frac{\log|z_x|}{|z_x|^{\frac13}}\sim r_x^{2}.
\end{align*}

Hence $\omega_{\operatorname{sf},\varepsilon}$ is $\operatorname{SOB}\left(2\right)$.

Now we compute the asymptotic cone of $M$. Note that the diameter of fiber $f^{-1}(z_x)$ is only of order $|\log|z_x||^{\frac12}$, hence we only need to compute the tangent cone of $g$. Define $
\Phi_\lambda:\alpha\mapsto \left(\frac{\alpha}{\alpha_\lambda}\right)^{-6}$, where $\alpha_\lambda\left|\log |\alpha_\lambda|\right|^{-\frac12}=\lambda$, then
\begin{align*}
\Phi_\lambda^*(\lambda^2g)=\frac{216\sqrt{3}bi|k|^2}{\varepsilon^2}\left|1-\frac{\log |\alpha|}{\log|\alpha_\lambda|}\right|d\alpha\wedge d\bar{\alpha}\stackrel{\lambda\to 0}{\longrightarrow} \frac{216\sqrt{3}bi|k(0)|^2}{\varepsilon^2}d\alpha\wedge d\bar{\alpha}.
\end{align*}Hence the tangent cone at infinity is a metric cone of angle $\frac{\pi}{3}$.

\begin{remark}

$\Omega_1\times\Omega_2\otimes s$ provides a global trivializing section of the trivial bundle $K_{X_1\times X_2}\otimes \mathcal{O}_{X_1\times X_2}(X)$,which implies that the invertible sheaf $(K_{X_1\times X_2}\otimes \mathcal{O}_{X_1\times X_2}(X))|_X$ over $X$ is trivial. However, we cannot conclude that $K_X$ is trivial, because $K_X$ is not well defined for a non-normal variety $X$. Moreover, the Poincar\'{e} residue $\Omega$ can only be defined on $M$, indicating $\Omega$ provides a global trivializing section of $K_M$. This observation is quite important for constructing complete Calabi-Yau metrics, because for all the cases where the fiber product variety $X$ is normal, $K_X$ is indeed trivial. To be more precise, if all the irreducible components of the two fibers $F_1$, $F_2$ are reduced (i.e. of multiplicity $1$, e.g. $\mathrm{I}_b,\mathrm{II},\mathrm{III},\mathrm{IV}$ type singular fibers), the resulting $X$ is normal. From \cite{kapustka2009fiber} and \cite{schoen1988fiber}, we know further that for $\mathrm{I}_n\times \mathrm{I}_m$, $\mathrm{I}_n\times \mathrm{II}$, $\mathrm{I}_n\times \mathrm{III}$, $\mathrm{I}_n\times \mathrm{IV}$, $\mathrm{II}\times \mathrm{II}$, $\mathrm{II}\times \mathrm{III}$, $\mathrm{III}\times \mathrm{III}$, $\mathrm{IV}\times \mathrm{IV}$ cases, we can carry out crepant resolution to $X$. Thus $\Omega$ can be lifted to a global holomorphic volume form on the resolution space. So, it is impossible in these cases to construct a complete Calabi-Yau metrics on $M$ by this method.

\end{remark}

\section{Perturbation of the semi-flat ansatz}\label{perturbation}

In this section, we aim to extend the semi-flat ansatz $\omega_{\operatorname{sf},\varepsilon}$ constructed in the previous section to an ambient metric in $M$. We will then employ the Tian-Yau-Hein package to perturb it into a genuine complete Calabi-Yau metric, ensuring that its asymptotic behavior resembles that of the ansatz. Compared to the elliptic fibration case studied by Hein, this approach may encounter additional topological obstructions.

Let $f: X \rightarrow \mathbb{P}^1$ be the abelian fibration constructed in Section \ref{non-isotrivial models}. Let $F=f^{-1}(p)$ be the singular fiber. Fix a small disk $\Delta\subset \mathbb{P}^1$ around $p$. Then $f$ has no singular fibers over $\Delta^*$. Write $\left.X\right|_{\Delta^*}=\left(\Delta_z^* \times \mathbb{C}^2\right)\big{/}\Lambda(z)$. We aim to find a metric $\omega$ on $X\backslash F$ and glue it with the semi-flat ansatz $\omega_{\operatorname{sf},\varepsilon}(\alpha)$ to produce a complete K{\"a}her metric $\omega_{\varepsilon}$, which coincide with $\omega_{\operatorname{sf},\varepsilon}(\alpha)$ near the singular fiber. As a result, the asymptotic behavior of $\omega_{\varepsilon}$ is described by that of $\omega_{\operatorname{sf},\varepsilon}(\alpha)$. Here 
\begin{equation}
\begin{aligned}
\omega_{\mathrm{sf}, \varepsilon}(\alpha)&=c(\alpha)i|g|^2 \frac{\operatorname{Im}\left(\bar{\tau}_1 \tau_2\right)\operatorname{Im}\left(\bar{\tau}_3 \tau_4\right)}{\varepsilon} d z \wedge d \bar{z}\\
&+\frac{i}{2} \frac{\varepsilon}{\operatorname{Im}\left(\bar{\tau}_1 \tau_2\right)}(d w_1-\Gamma^1 d z) \wedge(d \bar{w}_1-\bar{\Gamma}^1 d \bar{z})\\
&+\frac{i}{2} \frac{\varepsilon}{\operatorname{Im}\left(\bar{\tau}_3 \tau_4\right)}(d w_2-\Gamma^2 d z) \wedge(d \bar{w}_2-\bar{\Gamma}^2 d \bar{z}),
\end{aligned}
\end{equation}
where $c(\alpha)$ is chosen such that $\omega^3_{\operatorname{sf},\varepsilon}=\alpha\Omega\wedge \bar{\Omega}$. We need this parameter $\alpha$ to adjust the metric properly. Then we can apply the Tian-Yau-Hein's package to get a complete Calabi-Yau metric.

Next, we outline the obstructions encountered and the methods employed to find such a metric. Portions of this process are essentially identical to those described in Hein's thesis \cite{hein2010gravitational}, and thus we will omit redundant details.

\textbf{The First difficulty}: We need to prove a $\partial\bar{\partial}-$type lemma on $X\backslash F$. There are mainly two obstructions: One is purely topological and the other is analytic.

The first obstruction is to find a K{\"a}hler metric $\omega$ on $X\backslash F$ such that $[\omega]=[\omega_{\operatorname{sf},\varepsilon}(1)]$ in the de Rham cohomology group $H^2(X|_{\Delta^*},\mathbb{R})$. By invoking the Leray spectral sequence and following an argument identical to that in Hein's thesis \cite{hein2010gravitational}, we conclude that  $H_2(X|_{\Delta^*},\mathbb{Z})$ is generated by the six generators of $H_2(\hat{F},\mathbb{Z})\cong \mathbb{Z}^6$, where $\hat{F}$ is a general smooth fiber over $\Delta^*$ and those so called bad cycles.

\begin{definition}[Bad cycle]

Consider the restriction of the fibration to a unit disk centered at $z=0$, that is $f:X|_{\Delta^*}\to \Delta^*_z$. Let $\hat{F}$ be a general smooth fiber over a point in $\Delta^*$. Take a simple loop $\gamma\subset F$ such that $[\gamma]\in H_1(\hat{F},\mathbb{Z})$ is indivisible and invariant under the monodromy action. Lift another simple loop $\gamma'\subset \Delta^*$ to $X$, then move $\gamma$ around the puncture along the lifted loop. Now the union of the translates of $\gamma$ is a $T^2$ embedded in $f^{-1}(\gamma')$ and this is called a bad cycle.

\end{definition}

\begin{remark}

The bad cycle is formed by those non-vanishing cycles rotating around the puncture.

\end{remark}

In order to guarantee that $[\omega]=[\omega_{\operatorname{sf},\varepsilon}(1)]$ in $H^2(X|_{\Delta^*},\mathbb{R})$, we only need to show that 
\begin{itemize}

\item $\omega|_{\hat{F}}=\omega_{\operatorname{sf},\varepsilon}(1)|_{\hat{F}}$ for any fiber $\hat{F}$ over $\Delta^*$;

\item $\langle \omega,C\rangle=\langle \omega_{\operatorname{sf},\varepsilon}(1),C\rangle =0$ for any bad cycles $C$.

\end{itemize}

Note that $M=X\backslash F=X_1\backslash F_1\times_{\mathbb{P}^1}X_2\backslash F_2$ is fiber product of two rational elliptic surfaces. Let $\omega_1$, $\omega_2$ be two K{\"a}hler metrics on $X_1\backslash F_1$ and $X_2\backslash F_2$ respectively such that $\langle \omega_i,C_i\rangle=0$ and $\left[\omega_i|_{\Delta^*_i}\right]=\left[\omega_{\operatorname{sf},\varepsilon}^i(1)\right]$, where $\Delta^*_i$ denotes a neighborhood near the singular fiber of $X_i\to\mathbb{P}^1$, $C_i$ denotes a bad cycle on $X_i|_{\Delta^*_i}$ and
\begin{align*}
&\omega_{\operatorname{sf},\varepsilon}^1(1)=c(\alpha)i|g|^2 \frac{\operatorname{Im}\left(\bar{\tau}_1 \tau_2\right)}{\varepsilon} d z_1 \wedge d \bar{z}_1+\frac{i}{2} \frac{\varepsilon}{\operatorname{Im}\left(\bar{\tau}_1 \tau_2\right)}(d w_1-\Gamma^1 d z_1) \wedge(d \bar{w}_1-\bar{\Gamma}^1 d \bar{z}_1)\\
&\omega_{\operatorname{sf},\varepsilon}^2(1)=c(\alpha)i|g|^2 \frac{\operatorname{Im}\left(\bar{\tau}_3 \tau_4\right)}{\varepsilon} d z_1 \wedge d \bar{z}_1+\frac{i}{2} \frac{\varepsilon}{\operatorname{Im}\left(\bar{\tau}_3 \tau_4\right)}(d w_2-\Gamma^2 d z_1) \wedge(d \bar{w}_2-\bar{\Gamma}^2 d \bar{z}_1).
\end{align*}$\omega_1,\omega_2$ always exist and can be obtained by restriction of a global K{\"a}hler metric from $X_i$ (This relies heavily on the fact that $H_2(T^2,\mathbb{Z})=\mathbb{Z}$). Now take $\omega=\iota^*(\omega_1+ \omega_2)$, where $\iota:X\hookrightarrow X_1\times X_2$ is the inclusion map. By definition, we have $\left\langle\omega,C \right\rangle=\left\langle\omega_1+\omega_2,\iota_*C \right\rangle$, where $C$ is a bad cycle on $X|_{\Delta^*}$. Since the pairing kills all the torsion part, we only need to calculate $\iota_*C$ in $\mathbb{R}$-coefficients. According to the K{\"u}nneth formula, $H_2(X_1|_{\Delta^*}\times X_2|_{\Delta^*},\mathbb{R})\cong H_2(X_1|_{\Delta^*},\mathbb{R})\oplus H_1(X_1|_{\Delta^*},\mathbb{R})\otimes H_1(X_2|_{\Delta^*},\mathbb{R})\oplus H_2(X_2|_{\Delta^*},\mathbb{R})$. Since $X|_{\Delta^*}$ is fiber product of two elliptic fibrations, it's clear that the bad cycles of $X|_{\Delta^*}$ can be linearly generated by a loop on $\Delta^*$ times those invariant cycles in $\hat{F}_1$ and $\hat{F}_2$, where $\hat{F}_i$ is smooth fiber of $X_i\to \mathbb{P}^1$. Without losing generality, we assume that the bad cycle $C$ is generated by a invariant cycle in $\hat{F}_1$ and a loop on $\Delta^*$. According to the definition of the pushforward,  $\iota_*C$ lies in $H_2(X|_{\Delta^*},\mathbb{R})$.  This homology class decomposes into three parts:
\begin{itemize}

\item A bad cycle in $H_2(X_1|_{\Delta^*},\mathbb{R})$;

\item A nontrivial contribution from the tensor product $H_1(X_1|_{\Delta^*},\mathbb{R})\otimes H_1(X_2|_{\Delta^*},\mathbb{R})$;

\item A trivial component in $H_2(X_2|_{\Delta^*},\mathbb{R})$.

\end{itemize}

 Hence $\left\langle\omega_1+\omega_2,\iota_*C \right\rangle=0$, and $\omega$ satisfies our demands.

\begin{remark}

We call this obstruction topological because it essentially arises from $H^2(T^4,\mathbb{Z})\cong \mathbb{Z}^6$ and those bad cycles. Our proof relies heavily on the fiber product structure.

\end{remark}


The next obstruction comes from $H^{0,1}(X|_{\Delta^*})$. This is analogous to the classic $\partial\bar{\partial}$-lemma on a compact K{\"a}hler manifold. We have already proven that $\omega_{\operatorname{sf},\varepsilon}(1)-\omega$ is $d$-exact, and now we want to show that it's $\partial\bar{\partial}$-exact. This is in general unrealistic. However, we could translate $\omega_{\operatorname{sf},\varepsilon}(1)$ by pulling back along a section $T$, then require $T^*\omega_{\operatorname{sf},\varepsilon}(1)-\omega$ to be $\partial\bar{\partial}-$exact.

Write $\omega_{\operatorname{sf},\varepsilon}(1)-\omega=d\zeta$, where $\zeta$ is a smooth real 1-form on $\left.X\right|_{\Delta^*}$. Then $\xi:=\zeta^{0,1}$ satisfies $\bar{\partial} \xi=0$. From the Leray spectral sequence, we get a long exact sequence
\begin{equation*}
0\to H^1(\Delta^*,f_*\mathcal{O}_{X|_{\Delta^*}})\to H^1(X|_{\Delta^*},\mathcal{O}_{X|_{\Delta^*}})\to H^0(\Delta^*,R^1f_*\mathcal{O}_{X|_{\Delta^*}})\to H^2(\Delta^*,f_*\mathcal{O}_{X|_{\Delta^*}})\to \cdots.
\end{equation*}Since $\Delta^*$ is a Stein manifold, and $f_*\mathcal{O}_{X|_{\Delta^*}}=\mathcal{O}_{\Delta^*}$(fiber is a compact complex manifold), we have by Cartan's Theorem B that $H^i(\Delta^*,R^1f_*\mathcal{O}_{X|_{\Delta^*}})=0$, $\forall i\geq 1$. Thus
\begin{equation*}
H^{0,1}(X|_{\Delta^*})\cong H^1(X|_{\Delta^*},\mathcal{O}_{X|_{\Delta^*}})\cong H^0(\Delta^*,R^1f_*\mathcal{O}_{X|_{\Delta^*}}).
\end{equation*}The proper base change property implies that $(R^1f_*\mathcal{O}_{X|_{\Delta^*}})_z=H^{0,1}(f^{-1}(z))\cong \mathbb{C}^2$. Note that the fiberwise constant $(0,1)$-form $\eta_1=\frac{d\bar{v}_1}{ \operatorname{Im}\left(\bar{\tau}_1 \tau_2\right)}$ and  $\eta_2=\frac{d \bar{v}_2} { \operatorname{Im}\left(\bar{\tau}_1 \tau_2\right)}$ define two holomorphic sections of $R^1f_*\mathcal{O}_{X|_{\Delta^*}}$. If we restrict $\xi$ to $f^{-1}(z)$ and write the constant part of the restriction as $a_1(z) \eta_1(z)+a_2(z)\eta_2(z)$, we can identify the class of $\xi$ in $H^{0,1}\left(\left.X\right|_{\Delta^*}\right)$ with two holomorphic functions $a_1,a_2: \Delta^* \rightarrow \mathbb{C}$. 

For any two holomorphic functions $\sigma_1,\sigma_2: \Delta^* \rightarrow \mathbb{C}$ viewed as a section of $X$, the associated vertical translation $T$ is given by $T(z, v_1,v_2)=(z, v_1+\sigma_1(z),v_2+\sigma_2(z))$. By direct computation, we get
\begin{align*}
\begin{aligned}
& T^* \omega_{\mathrm{sf},\varepsilon}(1)-\omega_{\mathrm{sf},\varepsilon}(1)=d \tilde{\zeta}+\frac{i}{2} \frac{\varepsilon}{\operatorname{Im}\left(\bar{\tau}_1 \tau_2\right)}\sum\limits_{i=1}^2\left|\sigma_i^{\prime}(z)-\Gamma^i(z, \sigma_1(z),\sigma_2(z))\right|^2 d z \wedge d \bar{z} ;\\
& \tilde{\xi}=\tilde{\zeta}^{0,1}:=\frac{i}{2} \frac{\varepsilon}{\operatorname{Im}\left(\bar{\tau}_1 \tau_2\right)} \left\{\sum\limits_{i=1}^2\sigma_i(z)(d \bar{v_i}-\bar{\Gamma}^i(z, v) d \bar{z})\right\}
\end{aligned}
\end{align*}
Thus, if we choose $\sigma_i(z)=\frac{2 \sqrt{-1}}{\varepsilon} a_i(z)$, then $(\xi+\tilde{\xi})|_{f^{-1}(z)}=0\in H^{0,1}(f^{-1}(z))$, $\forall z\in \Delta^*$. By the definition of sheaf, we get $\xi+\tilde{\xi}=0\in H^{0,1}(X|_{\Delta^*})$, i.e. $\xi+\tilde{\xi}$ is $\bar{\partial}$-exact. Also note that the correction term $\frac{i}{2} \frac{\varepsilon}{\operatorname{Im}\left(\bar{\tau}_1 \tau_2\right)}\sum\limits_{i=1}^2\left|\sigma_i^{\prime}(z)-\Gamma^i(z, \sigma_1(z),\sigma_2(z))\right|^2 d z \wedge d \bar{z} $ is lifted from an open Riemann surface and thus has a smooth real potential.

In conclusion, we have proven:

\begin{proposition}[$\partial\bar{\partial}-$lemma]

There exists a holomorphic section $\tilde{\sigma}$ of $f$ over $\Delta^*$ and a smooth function $u_1:X|_{\Delta^*}\to \mathbb{R}$ such that $T^*\omega_{\operatorname{sf},\varepsilon}(1)-\omega=i\partial \bar{\partial u_1}$, where $T$ denotes the translation by $\tilde{\sigma}$ relative to the initially chosen section $\sigma$.

\end{proposition}

By using the $\partial\bar{\partial}$-lemma, we can glue the metrics $\omega$ and $\omega_{\operatorname{sf},\varepsilon}(1)$. However, direct gluing by cut-off functions may cause some troubles.

\textbf{The second difficulty:} $\omega$ and $\omega_{\operatorname{sf},\varepsilon}(1)$ may differ a lot in their overlapping region. Gluing by cut-off function could create large negative components, which makes the resulting form fail to be a metric.

\textbf{The third difficulty:} $\omega_{\operatorname{sf},\varepsilon}(1)$ is $\mathrm{SOB}(\beta)$ manifold, where $0<\beta<2$. Thus in order to apply the Tian-Yau-Hein's package, we need an integrable condition
\begin{equation*}
\int_{X\backslash F}\left(\omega^3_{\operatorname{sf},\varepsilon}(\alpha)-\alpha i\Omega\wedge \bar{\Omega}\right)=0.
\end{equation*}

The second and third difficulties should be handled simultaneously. Nevertheless, for the second and third difficulties, Hein's proof (see \cite{hein2010gravitational}, Chapter 5, Claims 2, 3, and 4) applies to our example after modifications. For completeness, we still provide the detailed proof.

The global ansatz metric $\omega_0 = \omega_0(\alpha, t)$ we construct will depend on two parameters, $\alpha$ and $t$. The idea is that the two metrics we wish to glue are sufficiently close along the fibers, but their difference in the base term $dz \wedge d\bar{z}$ may violate positivity. To compensate, we pull back a sufficiently large $dz \wedge d\bar{z}$ term from the base (corresponding to the parameter $t$). On the other hand, we use the parameter $\alpha$ to achieve the desired integrability condition.

We can express the difference between $\omega_{\text{sf},\varepsilon}(\alpha)$ and $\omega_{\text{sf},\varepsilon}(1)$ as a potential function, namely,

\begin{equation}\label{diff}
    \omega_{\text{sf},\varepsilon}(\alpha)-\omega_{\text{sf},\varepsilon}(1)=i(\alpha-1)|k(z)|^2\frac{\operatorname{Im} (\tau_1\bar{\tau}_2) \operatorname{Im} (\tau_3\bar{\tau}_4)}{\varepsilon}\frac{dz\wedge d\bar{z}}{|z|^6}=(\alpha-1)i\partial\bar{\partial}u.
\end{equation}

Then by the $\partial\bar{\partial}$-lemma, we have

\begin{equation*}
T^*\omega_{\text{sf},\varepsilon}(\alpha) = \omega + i\partial\bar{\partial}u_\alpha, \quad u_\alpha := u_1 + (\alpha - 1)u. 
\end{equation*}

Here, $u_1$ and $u$ are not uniquely determined, but we fix an arbitrary choice for now and later adjust it using a harmonic function $v$.

Fix a sufficiently small parameter $r$, a sufficiently small parameter $s$ depending on $r$, and a positive constant $C_0$ satisfying $C_0 r < 1$ and $C_0 s < r$. On $\mathbb{P}^1$, fix a $(1,1)$-form $\beta$ such that $0 \leq \beta \leq dz \wedge d\bar{z}$, $\mathrm{supp}(\beta) \subset \Delta(r+3s)\setminus \Delta(r)$, and $\beta \equiv dz \wedge d\bar{z}$ on $\Delta(r+2s)\setminus \Delta(r+s)$. For $\alpha > 0$ and $t > 0$, we define

\begin{equation*}
\omega_0(\alpha, t) := 
\begin{cases} 
\omega + t\beta + i\partial\bar{\partial}(\psi\tilde{u}_\alpha) & \text{outside } M|_{\Delta(r)}; \\
\omega + t\beta + i\partial\bar{\partial}u_\alpha & \text{on } \Delta(r+s)^*,
\end{cases}
\end{equation*}

where
\begin{equation*}
\tilde{u}_\alpha := u_1 + (\alpha - 1)(u - v) \quad \text{in } \Delta(r+3s) \setminus \Delta(r),
\end{equation*}
$v$ is the harmonic function on $\Delta(r+3s) \setminus \Delta(r)$ with the same boundary values as $u$, and $\psi = \psi(r,s)$ is a cutoff function satisfying $0 \leq \psi \leq 1$, $\mathrm{supp}(\psi) \subset \Delta(r+2s)$, $s|\psi_z| + s^2|\psi_{z\bar{z}}| \leq C_0$, and $\psi \equiv 1$ on $\Delta(r+s)$. Note that $\omega_0(\alpha, t)$ coincides with $\omega$ outside $M|_{\Delta(r+3s)}$, and equals exactly $T^*\omega_{\text{sf},\varepsilon}(\alpha)$ on $\Delta(r)^*$.

We then define:
\begin{equation*}
    I(\alpha, t) := \int_M \left( \omega_0(\alpha, t)^3 - \alpha i\Omega \wedge \bar{\Omega} \right) = 0.
\end{equation*}
Our goal is to choose suitable $\alpha$ and $t$ so that $\omega_0(\alpha, t)$ is positive definite while $I(\alpha, t) = 0$.

The key observation is that $I(\alpha, t)$ is bilinear in the two variables: the terms in $\omega_0(\alpha, t)^3$ involving $t$ and $\alpha$ are in fact all of type $dz \wedge d\bar{z}$, so higher-order or mixed terms do not appear. Naively speaking, if we fix one variable (say $\alpha$) and vary the other (say $t$), then as long as the coefficient in front of the varying variable is non-zero (which generically holds), we can choose an appropriate value for that parameter. However, the issue is that $\omega_0(\alpha, t)$ is only positive definite when $t$ is sufficiently large. Therefore, we need to carefully estimate the range of $t$ depending on $\alpha$, as well as the coefficients of these two parameters in $I(\alpha, t)$. For this, we require an analytic lemma:

\begin{lemma}[\cite{hein2010gravitational}, Chapter 5, Claim 2]\label{normalization}
    For a sufficiently small parameter $r$, a sufficiently small parameter $s$ depending on $r$, and a positive constant $C_0$ satisfying $C_0 r < 1$ and $C_0 s < r$, if $v$ is the harmonic function on $\Delta(r+3s) \setminus \Delta(r)$ that agrees with $u$ on the boundary, then we have the estimate

    \begin{equation*}
    \sup_{\Delta(r+2s)\setminus\Delta(r+s)} \left( s^{-2}|u-v| + s^{-1}|\nabla (u-v)| \right) \leq C_0 \sup_{\Delta(r+3s)\setminus\Delta(r)} u_{z\bar{z}}.
    \end{equation*}

\end{lemma}

Now we show that as long as $t$ is sufficiently large (depending on $r, s, \alpha$), $\omega_0(\alpha, t)$ is indeed a metric on the entire manifold $M$. More precisely, we have:

\begin{proposition}\label{positivity}
    For a sufficiently small parameter $r$, and a sufficiently small parameter $s$ depending on $r$, there exist constants $C_0, C_1$ depending on $r, s$ but not on $\alpha$, such that $C_0 r < 1$, $C_0 s < r$, and if
    $$
    t > C_1 + C_0|\alpha - 1| \sup_{\Delta(r+3s)\setminus\Delta(r)} u_{z\bar{z}},
    $$
    then
    $$
    \omega_0(\alpha, t) \geq \frac{1}{2}(\omega + \psi i\partial\bar{\partial}u_\alpha) > 0.
    $$
\end{proposition}

\begin{proof}
    It suffices to verify positivity in the region $\Delta(r+2s) \setminus \Delta(r+s)$. We expand the expression for $\omega_0(\alpha, t)$:
\begin{align*}
\omega_0(\alpha, t) &= \omega + t\beta + \psi i\partial\bar{\partial}u_\alpha + \Big[\psi_z u_{1,\bar{w_1}} idz \wedge d\bar{w_1} + \psi_{\bar{z}} u_{1,w_1} idw_1 \wedge d\bar{z} \\
&\quad + \psi_z u_{1,\bar{w_2}} idz \wedge d\bar{w_2} + \psi_{\bar{z}} u_{1,w_2} idw_2 \wedge d\bar{z} \\
&\quad + \big\{ \psi_{z\bar{z}} \tilde{u}_\alpha + \psi_{\bar{z}} \tilde{u}_{\alpha,z} + \psi_z \tilde{u}_{\alpha,\bar{z}} \big\} idz \wedge d\bar{z} \Big].
\end{align*}

Since $\omega + i\partial\bar{\partial}u_\alpha = T^*\omega_{\text{sf},\varepsilon}(\alpha) > 0$ and $\omega$ itself is strictly positive, their convex combination $\omega + \psi i\partial\bar{\partial}u_\alpha$ is also strictly positive.

We first control the $dz \wedge d\bar{z}$ term. The derivatives of $\psi$ are controlled by its definition. As for $\tilde{u}_\alpha = u_1 + (\alpha - 1)(u - v)$, the $u_1$-term is independent of $\alpha$ and can be bounded by $C_1$. The contribution from $(\alpha - 1)(u - v)$ is controlled using Lemma~\ref{normalization}, leading to the condition
$$
t > C_1 + C_0|\alpha - 1| \sup_{\Delta(r+3s)\setminus\Delta(r)} u_{z\bar{z}}.
$$

For the remaining terms inside the brackets, they correspond to the matrix
$$
\begin{pmatrix}
0 & \psi_z u_{1,\bar{w_1}} & \psi_z u_{1,\bar{w_2}} \\
\psi_{\bar{z}} u_{1,w_1} & 0 & 0 \\
\psi_{\bar{z}} u_{1,w_2} & 0 & 0
\end{pmatrix}.
$$

The term $t\beta$ adds a $t$ to the top-left position of the matrix, while $\frac{1}{2}(\omega + \psi i\partial\bar{\partial}u_\alpha)$ contributes a positive function $\delta(z)$ to the remaining diagonal positions. To ensure the full matrix is positive definite, it suffices to ensure all its leading principal minors are positive. A computation shows that this holds provided we choose
\begin{equation}\label{control}
    t > \frac{\sup |\psi_z u_{1, \bar{\omega}_1}| + \sup |\psi_z u_{1, \bar{\omega}_2}|}{\inf \delta(z)}.
\end{equation}

In fact, $\inf \delta(z)$ admits a positive lower bound independent of $\alpha$, because $\omega + \psi i\partial\bar{\partial}u_\alpha + \beta$ is positive definite for all $\alpha \geq 0$: $\omega$ is positive definite, $\omega + i\partial\bar{\partial}(u_1 - u)$ is semi-positive, and $\beta$ completes it to a positive definite form via convex combination. In particular, the cofactors corresponding to $d\omega_1 \wedge d\bar{\omega}_1$ and $d\omega_2 \wedge d\bar{\omega}_2$ admit uniform positive lower bounds $K(z)$, achieved at $\alpha = 0$. A short calculation then shows that choosing $\delta(z) < \det(\omega + \psi i\partial\bar{\partial}u_0 + \beta)/4K(z)$ suffices. Therefore, the right-hand side of \eqref{control} has an upper bound independent of $\alpha$, which we may absorb into the constant $C_1$.
\end{proof}

Next, we choose suitable parameters $r, s, \alpha$, and $t$ to address the "integrability condition".

According to the above estimates, if we fix appropriate constants $r, s, C_0, C_1$, and define the function
$$
t(\alpha) = C_1 + C_0(\alpha - 1)\sup_{\Delta(r+3s)\setminus\Delta(r)} |u_{z\bar{z}}|,
$$
then $\omega_0(\alpha, t(\alpha))$ is positive definite for every $\alpha \geq 0$, and $I(\alpha, t(\alpha))$ is linear in $\alpha$.

When $\alpha = 0$, by definition we have
$$
I(0,t(0)) = \int_M \omega_0(0,C_1)^3.
$$
From the previous discussion, we know that $\omega_0(0,C_1)$ is semi-positive and strictly positive outside $\Delta(r+3s)$ (since it coincides with $\omega$), so $I(0,t(0)) > 0$. Therefore, as long as the derivative of $I(\alpha, t(\alpha))$ with respect to $\alpha$ is negative, we can find a suitable $\alpha$ satisfying the integrability condition.

Indeed, on $\Delta(r)^*$, we have $\omega_0(\alpha,t) = T^*\omega_{\mathrm{sf},\varepsilon}(\alpha)$, and since
$$
T^*\omega_{\mathrm{sf},\varepsilon}(\alpha)^3 = \alpha i\Omega \wedge \bar{\Omega},
$$
the corresponding terms in $I(\alpha, t)$ cancel out on this region. Thus, we only need to compute the derivative outside $\Delta(r)$. Moreover, outside $\Delta(r+3s)$, $\omega_0(\alpha, t)$ coincides with $\omega$, which is independent of $\alpha$. Using the definition of $\omega_0(\alpha, t(\alpha))$, we compute:

\begin{align*}
    I'(\alpha, t(\alpha)) &= \int_{\Delta(r+3s)\setminus \Delta(r)} (\omega + i\partial\bar{\partial}(\psi u_1))^2 \wedge \left[ C_0\sup|u_{z\bar{z}}|\beta + i\partial\bar{\partial}(\psi(u - v)) \right] \\
    &\quad - \int_{\mathbb{P}^1 \setminus \Delta(r)} i\Omega \wedge \bar{\Omega} \\
    &\leq \int_{\Delta(r+3s)\setminus \Delta(r)} (4\omega + 2i\partial\bar{\partial}u_1)^2 \wedge 2C_0\sup|u_{z\bar{z}}|dz \wedge d\bar{z} - \int_{\mathbb{P}^1 \setminus \Delta(r)} i\Omega \wedge \bar{\Omega}.
\end{align*}

Here, the inequality follows from similar estimates used in the proof of Proposition~\ref{positivity}: $\beta \leq dz \wedge d\bar{z}$, $i\partial\bar{\partial}(\psi(u - v)) \leq C_0\sup|u_{z\bar{z}}|dz \wedge d\bar{z}$, and
$$
\omega + i\partial\bar{\partial}(\psi u_1) \leq 2(\omega + \psi i\partial\bar{\partial}u_1 + C_1\beta)
\leq 4\omega + 2\psi i\partial\bar{\partial}u_1 + 2C_1\beta.
$$

Now observe that once $r$ is fixed sufficiently small, all terms in the above estimate for $I'(\alpha, t(\alpha))$, except for $C_0$ and the domain of integration, are independent of $s$. Moreover, from earlier estimates, $C_0$ is controlled by the condition $C_0 r < 1$ (possibly up to fixed multiplicative and additive constants). Therefore, by choosing $s$ sufficiently small, we can ensure that
$$
I'(\alpha, t(\alpha)) < 0,
$$
which implies that there exists a unique $\alpha_0 > 0$ such that $I(\alpha_0, t(\alpha_0)) = 0$. This completes our construction of the glued metric satisfying the integrability condition.

Finally, from the above discussion, we also see that if $\alpha > \alpha_0$, then $I(\alpha, t(\alpha)) < 0$, and we may further increase $t$ to obtain another set of admissible parameters.

To conclude, we have

\begin{theorem}

There exist a holomorphic section $\tilde{\sigma}$ over $\Delta^*$, concentric disks $\Delta^{\prime} \subset \Delta^{\prime \prime} \subset \Delta^{\prime \prime \prime} \subset$ $\Delta$, a $(1,1)$-form $\beta \geq 0$ on $\mathbb{P}^1$ such that $\operatorname{supp}(\beta) \subset \Delta^{\prime \prime \prime} \backslash \Delta^{\prime}$, and a constant $\alpha_0>0$. For all $\alpha>\alpha_0$, there exist functions $u_\alpha^{\mathrm{int}} \in C^{\infty}\left(f^{-1}\left(\mathbb{P}^1 \backslash \Delta^{\prime}\right), \mathbb{R}\right)$ and $u_\alpha^{\mathrm{ext }} \in C^{\infty}\left(f^{-1}\left(\Delta^{\prime \prime} \backslash\{0\}\right), \mathbb{R}\right)$ such that
\begin{itemize}

\item their complex Hessians coincide over $\Delta^{\prime \prime} \backslash \Delta^{\prime}$;

\item for all $t>t_\alpha$,  $\omega_0(\alpha, t)\coloneqq \omega+t f^* \beta+i \partial \bar{\partial} u_\alpha^{\mathrm{int}, \mathrm{ ext }}$ is a positive definite closed $(1,1)$-form on $X\backslash F$;

\item $\omega_0(\alpha, t)=\omega$ over $\mathbb{P}^1 \backslash \Delta^{\prime \prime \prime}$;

\item $\omega_0(\alpha, t)=T^* \omega_{\mathrm{sf}}(\alpha)$ over $\Delta^{\prime} \backslash\{0\}$, where $T$ denotes vertical translation by $\tilde{\sigma}$ relative to $\sigma$.

\end{itemize}
Furthermore, $\int_M\left(\omega_0(\alpha, t)^3-\alpha i\Omega \wedge \bar{\Omega}\right)=0$ holds for exactly one $t>t_\alpha$.

\end{theorem}

Therefore, we could apply Tian-Yau-Hein's package (Theorem \ref{TYH package}) to get a complete noncompact Calabi-Yau metric on $X\backslash F$. The asymptotic behavior of $\omega_{\operatorname{CY}}$ resembles $\omega_{\operatorname{sf},\varepsilon}(\alpha)$ because of Theorem \ref{conv to const}. This concludes the proof of Theorem \ref{main 2}.

\section{Further discussion}\label{Further discussion}

We list some possible further problems in this direction.

\noindent \textbf{Completing isotrivial model:} It is important to note that in the case of elliptic fibrations, all ALG-type gravitational instantons can be derived from isotrivial models. For instance, one can consider the quotient $(\mathbb{P}^1_{[t:s]}\times E)/\mathbb{Z}_k$, where $k=2,3,4,6$, as described in Hein's thesis \cite{hein2010gravitational}. By performing Kodaira's canonical resolution on this quotient, an isotrivial rational elliptic surface is obtained, thereby enumerating all possible ALG gravitational instantons.

However, in the context of abelian surface fibrations, smooth relatively minimal models do not always exist. Even when such models do exist, as discussed in the final paragraph of Section \ref{isotrivial models}, finding an appropriate gluing metric remains challenging. Consequently, the model $(\mathbb{P}^1\times A)/\mathbb{Z}_k$ does not suffice for all cases of finite monodromy. This necessitates exploring alternative methods for constructing isotrivial models.


\noindent \textbf{More general construction:} Our second construction relies heavily on the fiber product structure. According to Kenji Ueno \cite{ueno1971fiber} and \cite{ueno1972fiber}, there are many abelian surface fibrations that are not fiber product of elliptic fibration. One possible method is to consider a genus two curve fibration over $\mathbb{P}^1$ and then take the Jacobian fibration which is automatically a polarized abelian surface fibration. However, for this general construction, we may encounter two difficulties. The first difficulty is to find a global model whose canonical bundle is negative multiple of a fiber. The second difficulty arises from gluing procedure.

\noindent \textbf{Relation with the Tian-Yau metric:} Recall that the volume growth of the Tian-Yau metric in complex dimension three is of $\frac32$ order, which is the same as that in our construction of the $\mathrm{I}^*_{b_1}\times \mathrm{I}^*_{b_2}$ case. It's natural to ask whether these two constructions are related. We also remark that the $\mathrm{ALH}^*$ gravitational instanton (corresponding to the $\mathrm{I}_b$-type singular fiber) is equivalent to the Tian-Yau metric as discussed in \cite{hein2021asymptotically}.

\appendix

\section{Tian-Yau-Hein's package}\label{Tian-Yau-Hein's package}

For readers' convenience, we recall the Tian-Yau-Hein's package in this appendix.

\begin{definition}[$\operatorname{SOB}(\beta)$]\label{SOB}

Let $(M,g)$ be a complete non-compact Riemannian manifold with real dimension at least $3$. Let $\beta\in \mathbb{R}^+$. We say $(M,g)$ satisfies $\operatorname{SOB}(\beta)$ condition if there exists a fixed $x_0\in M$ and a positive constant $C$ such that

(1) $\operatorname{Vol}_g(B(x_0,R))\leq CR^\beta$, $\forall R>C$;

(2) $\operatorname{Vol}_g\left(B\left(x,\left(1-\frac 1C\right)d_g(x,x_0)\right)\right)\geq \frac1C d_g(x,x_0)^\beta$;

(3) $\operatorname{Ric}(x)\geq -Cd_g(x,x_0)^{-2}$, $\forall x\in M$;

(4) For any $D>C$, any two points $x,y\in M$ with $d_g(x_0,x)=d_g(x_0,y)=D$ can be joined by a curve in the annulus $A\left(x_0,\frac1C D,CD\right)=\left\{x\in M:\frac1C D<d_g(x_0,x)<CD\right\}$

\end{definition}

\begin{definition}[$C^{k,\alpha}$-atlas]

Let $\left(M, \omega_0\right)$ be a complete Kähler manifold. A $C^{k, \alpha}$ quasi-atlas for $\left(M, \omega_0\right)$ is a collection $\left\{\Phi_x: x \in A\right\}, A \subset M$, of holomorphic local diffeomorphisms $\Phi_x: B \rightarrow M, \Phi_x(0)=x$, from $B=B(0,1) \subset \mathbb{C}^m$ into $M$ which extend smoothly to the closure $\bar{B}$, and such that there exists $C \geq 1$ with $\operatorname{inj}\left(\Phi_x^* g_0\right) \geq \frac{1}{C}, \frac{1}{C} g_{\mathbb{C}^m} \leq \Phi_x^* g_0 \leq C g_{\mathbb{C}^m}$, and $\left\|\Phi_x^* g_0\right\|_{C^{k, \alpha}\left(B, g_{\mathbb{C}^m}\right)} \leq C$ for all $x \in A$, and such that for all $y \in M$ there exists $x \in A$ with $y \in \Phi_x(B)$ and $\operatorname{dist}_{g_0}\left(y, \partial \Phi_x(B)\right) \geq \frac{1}{C}$.

\end{definition}

\begin{lemma}[Tian-Yau \cite{tian1990complete}, Proposition 1.2]\label{QuasiAtlas}

If $|\mathrm{Rm}| \leq C$, then there exists a quasi-atlas which is $C^{1, \alpha}$ for every $\alpha$. If moreover $\sum_{i=1}^k \mid \nabla^i$ Scal $\mid \leq C$, then this quasi-atlas is even $C^{k+1, \alpha}$.

\end{lemma}

\begin{theorem}[Tian-Yau-Hein's Package\cite{hein2010gravitational}, Proposition 6.1]\label{TYH package}

Let $(M,\omega_0)$ be a complete noncompact K{\"a}hler manifold with a $C^{3,\alpha}$ quasi-atlas which satisfies $\mathrm{SOB}(\beta)$ for some $\beta>0$. 

Let $f \in C^{2, \alpha}(M)$ satisfy $|f| \leq C r^{-\mu}$ on $\{r>1\}$ for some $\mu>2$. If $\beta \leq 2$, then assume in addition that $\int_M\left(e^f-1\right) \omega_0^m=0$. Then there exist $\bar{\alpha} \in(0, \alpha]$ and $u \in C^{4, \bar{\alpha}}(M)$ such that $\left(\omega_0+i \partial \bar{\partial} u\right)^m=e^f \omega_0^m$. If $\beta \leq 2$, then moreover $\int_M|\nabla u|^2 \omega_0^m<\infty$. Independent of the value of $\beta$, if in addition $f \in C_{\mathrm{loc}}^{k, \bar{\alpha}}(M)$ for some $k \geq 3$, then all such solutions $u$ belong to $C_{\mathrm{loc}}^{k+2, \bar{\alpha}}(M)$.

\end{theorem}

\begin{definition}[$\operatorname{HMG}(\lambda, k, \alpha)$]

Let $(M,g)$ be a complete noncompact Riemannian manifold with real dimension at least $3$. $\left(M^n, g\right)$ is called $\operatorname{HMG}(\lambda, k, \alpha)$, $\lambda \in[0,1]$, $k \in \mathbb{N}_0$, $\alpha \in(0,1)$, if there exist $x_0 \in M$ and $C \geq 1$ such that

(1) for every $x \in M$ with $r(x) \geq C$, there exists a local diffeomorphism $\Phi_x$ from the unit ball $B \subset \mathbb{R}^n$ into $M$ such that $\Phi_x(0)=x$ and $\Phi_x(B) \supset B\left(x, \frac{1}{C} r(x)^\lambda\right)$;

(2) $h:=r(x)^{-2 \lambda} \Phi_x^* g$ satisfies $\operatorname{inj}(h) \geq \frac{1}{C}, \frac{1}{C} g_{\text {euc }} \leq h \leq C g_{\text {euc }}$, and $\left\|h-g_{\text {euc }}\right\|_{C^{k, \alpha}\left(B, g_{\mathrm{euc}}\right)} \leq C$.

\end{definition}

On a complete HMG Riemannian manifold, we can define the weighted H\"{o}lder norm. For any smooth weight function 
$\varphi:\left[0,\infty\right)\to\left[1,\infty\right)$, define: 
\begin{equation*}
    \|u\|_{\varphi, l, \gamma}:=\|u\|_{C^{l,\gamma}(B(x_0,2C),g)}+ \sup\{\varphi(r(x))\|u\circ\Phi_x\|_{C^{l,\gamma}(B,g_{\mathrm{euc}}): r(x)>C}\}
\end{equation*}
where $B$ and $\Phi_x$ are the same as in the definition of HMG manifolds.

\begin{lemma}[\cite{hein2010gravitational}, Lemma 4.6]

A complete Kähler manifold with $|\mathrm{Rm}|+\sum_{i=1}^k r^{i \lambda}\left|\nabla^i \mathrm{Scal}\right| \leq C r^{-2 \lambda}$ for some $k \in \mathbb{N}_0$ and $\lambda \in[0,1]$ is $\operatorname{HMG}(\lambda, k+1, \alpha)$ for every $\alpha \in(0,1)$.

\end{lemma}

\begin{theorem}[Convergence to constant\cite{hein2010gravitational}, Proposition 4.8]\label{conv to const}

Let $\left(M^m, \omega_0\right)$ be a complete Kähler manifold and let $u, f \in C^{\infty}(M)$ be such that $\sup \left|\nabla^i u\right|+\sup \left|\nabla^i f\right|<\infty$ for all $i \in \mathbb{N}$, and $\left(\omega_0+i \partial \bar{\partial} u\right)^m=e^f \omega_0^m$. 

(1) Assume $M$ is $\operatorname{SOB}(\beta)$ with $\beta \in(0,2]$, and $r(x)^\kappa|B(x, 1)| \rightarrow \infty$ as $r(x) \rightarrow \infty$ for every fixed $\kappa>0$. If $\int_M|\nabla u|^2 \omega_0^m<\infty$ and $|f| \leq C r^{-\beta-\varepsilon}$ for some $\varepsilon>0$, then $$\sup _{B(x, 1)}\left|u-u_{B(x, 1)}\right| \leq C r(x)^{-\delta}$$ for some $\delta>0$ and all $x \in M$.

(2) Assume $M$ is $\operatorname{HMG}(\lambda, k+1, \alpha)$ and that $M$ has at most linear diameter growth in the sense that for all $s \geq C$ and $x_1, x_2 \in M$ with $r\left(x_1\right)=r\left(x_2\right)=s$, there exists a path $\gamma: x_1 \rightarrow x_2$ with length $(\gamma) \leq C s$ and $s-\frac{1}{C} s^\lambda \leq r \circ \gamma \leq C$ s.

Let $\varphi:[0,\infty)\to [1,\infty)$ be a smooth weight function such that $\varphi(t)(1+t)^{-\delta}$ is non-decreasing for some $\delta>0$, and $\varphi\left(t-\frac{1}{2} t^\lambda\right) \geq \frac{1}{C} \varphi(t)$. If $\sup _x \varphi(r(x)) r(x)^{1-\lambda} \sup _{B(x, 1)}\left|u-u_{B(x, 1)}\right|<\infty$ and $\|f\|_{\psi, k, \alpha}<\infty$ with $\psi(t):=(1+t)^{1+\lambda} \varphi(t)$, then $$\|u-\bar{u}\|_{\varphi, k+2, \alpha}<\infty$$ for some constant $\bar{u} \in \mathbb{R}$.

\end{theorem}

\bibliographystyle{plain}
\bibliography{reference}

\begin{thebibliography}{10}

\bibitem{crauder1994minimal}
Bruce Crauder and David~R Morrison.
\newblock Minimal models and degenerations of surfaces with {K}odaira number zero.
\newblock {\em Transactions of the American Mathematical Society}, 343(2):525--558, 1994.

\bibitem{hein2010gravitational}
Hans-Joachim Hein.
\newblock {\em On gravitational instantons}.
\newblock Princeton University, 2010.

\bibitem{hein2021asymptotically}
Hans-Joachim Hein, Song Sun, Jeff Viaclovsky, and Ruobing Zhang.
\newblock Asymptotically {C}alabi metrics and weak {F}ano manifolds.
\newblock {\em arXiv preprint arXiv:2111.09287}, 2021.

\bibitem{hein2015remarks}
Hans-Joachim Hein and Valentino Tosatti.
\newblock Remarks on the collapsing of torus fibered {C}alabi--{Y}au manifolds.
\newblock {\em Bulletin of the London Mathematical Society}, 47(6):1021--1027, 2015.

\bibitem{joyce2001quasi}
Dominic~D. Joyce.
\newblock Quasi-{ALE} metrics with holonomy {SU}(m) and {Sp}(m).
\newblock {\em Annals of Global Analysis and Geometry}, 19:103--132, 2001.

\bibitem{kapustka2009fiber}
Grzegorz Kapustka and Micha{\l} Kapustka.
\newblock Fiber products of elliptic surfaces with section and associated {K}ummer fibrations.
\newblock {\em International Journal of Mathematics}, 20(04):401--426, 2009.

\bibitem{kodaira1963compact}
Kunihiko Kodaira.
\newblock On compact analytic surfaces: Ii.
\newblock {\em Annals of Mathematics}, 77(3):563--626, 1963.

\bibitem{lye2019stable}
J{\o}rgen~Olsen Lye.
\newblock {\em Stable geodesics on a K3 surface}.
\newblock PhD thesis, Dissertation, Universit{\"a}t Freiburg, 2019, 2019.

\bibitem{oguiso1997note}
Keiji Oguiso.
\newblock A note on moderate abelian fibrations.
\newblock {\em Contemporary Mathematics}, 207:101--118, 1997.

\bibitem{roan1996minimal}
Shi-Shyr Roan.
\newblock Minimal resolutions of {G}orenstein orbifolds in dimension three.
\newblock {\em Topology}, 35(2):489--508, 1996.

\bibitem{schoen1988fiber}
Chad Schoen.
\newblock On fiber products of rational elliptic surfaces with section.
\newblock {\em Mathematische Zeitschrift}, 197:177--199, 1988.

\bibitem{tian1990complete}
Gang Tian and Shing-Tung Yau.
\newblock Complete {K}{\"a}hler manifolds with zero {R}icci curvature {I}.
\newblock {\em Journal of the American Mathematical Society}, 3(3):579--609, 1990.

\bibitem{ueno1971fiber}
Kenji Ueno.
\newblock On fiber spaces of normally polarized abelian varieties of dimension 2, {I}.
\newblock {\em Journal of the Faculty of Science. University of Tokyo}, 17:37--95, 1971.

\bibitem{ueno1972fiber}
Kenji Ueno.
\newblock On fiber spaces of normally polarized abelian varieties of dimension 2, {II}.
\newblock {\em Journal of the Faculty of Science. University of Tokyo}, 17:163--199, 1972.

\bibitem{wang2022ricci}
Yuanqi Wang.
\newblock Ricci-flat manifolds of generalized {ALG} asymptotics.
\newblock {\em arXiv preprint arXiv:2212.11267}, 2022.

\bibitem{yau1978ricci}
Shing-Tung Yau.
\newblock On the ricci curvature of a compact {K}{\"a}hler manifold and the complex {M}onge-{A}mp{\'e}re equation, i.
\newblock {\em Communications on pure and applied mathematics}, 31(3):339--411, 1978.

\end{thebibliography}


\end{document}